%% file: main_new.tex
\documentclass[11pt]{article}
\include{preamble}

\title{Gradient descent avoids strict saddles\\with a simple line-search method too}

\author{
  Andreea-Alexandra Musat and Nicolas Boumal\thanks{Ecole Polytechnique F\'ed\'erale de Lausanne (EPFL), Institute of Mathematics,
	  \texttt{\{andreea.musat, nicolas.boumal\}@epfl.ch}.}
}
\date{\today}

\begin{document}

\maketitle

\begin{abstract}
	It is known that gradient descent (GD) on a $C^2$ cost function generically avoids strict saddle points when using a small, constant step size.
	However, no such guarantee existed for GD with a line-search method.
	We provide one for a modified version of the standard Armijo backtracking method with generic, arbitrarily large initial step size.
	The proof underlines the role of the Luzin $N^{-1}$ property for the iteration maps, and allows to forgo the habitual Lipschitz gradient assumption.
	
	We extend this to the Riemannian setting (RGD), assuming the retraction is real analytic (though the cost function still only needs to be $C^2$).
	In closing, we also improve guarantees for RGD with a constant step size in some scenarios.
\end{abstract}
	


\section{Introduction}

Many standard minimization algorithms may, in theory, converge to saddle points.
Fortunately, this almost never happens in practice because, while a nearby saddle point may slow progress, it rarely attracts iterates. 
To formalize this empirical observation, \citet{lee2016gradient} showed that gradient descent with a fixed step size $\alpha$ avoids 
strict saddles from almost all initializations.
Their early result assumes the saddles are isolated and the step size satisfies $\alpha < 1/L$,
where $L$ is the Lipschitz constant of the gradient.
\citet{panageas2017gradient} soon extended this result to non-isolated saddles under the same step size condition.

These results have been generalized to other algorithms beyond gradient descent; see notably their joint paper \citep{lee2019strictsaddles} and further \emph{related work} below.
However, such extensions generally apply to algorithms of the form $x_{t+1} = g(x_t)$ with a generic initial point $x_0$, where $g$ is a sufficiently regular iteration map.
Some works extend the analysis to algorithms of the form $x_{t+1} = g_t(x_t)$, where each $g_t$ satisfies certain regularity conditions as well as additional conditions that link them together.

In both cases, the sequence of iteration maps is fixed in advance:
the algorithm does not adapt to the iterates generated so far.
This limitation stems from the proofs' reliance on the Center Stable Manifold Theorem (CSMT),
a fundamental result in dynamical systems.
The CSMT does not allow for the kind of adaptivity that would dynamically adjust the iteration map based on the sequence of iterates.

Still, the empirical observation that ``reasonable algorithms typically avoid saddle points'' extends well beyond gradient descent with a constant step size and other such simple methods.
It is natural to wonder whether this more general picture too can be formalized.
To begin with, we ask:
\begin{center}
	\emph{Does gradient descent with a standard line-search method avoid strict saddle points?}
\end{center}
\noindent
Surprisingly, even this falls outside the scope of existing results.
We provide a positive answer for a modified backtracking line-search (Algorithm~\ref{algo:bkt-stab}).
In so doing, we also improve results for constant step sizes, in both the Euclidean and the Riemannian settings.

\subsection{Saddle avoidance with line-search in $\Rn$}

Consider minimizing $f \colon \Rn \to \reals$ by iterating $x_{t+1} = x_t - \alpha_t \nabla f(x_t)$.
A standard backtracking line-search sets an initial step size $\bar{\alpha} > 0$, a decay parameter $\tau \in (0, 1)$ and a tolerance parameter $r \in (0, 1)$, then selects the step size $\alpha_t$ as the first value in the geometric sequence $\bar{\alpha}, \tau \bar{\alpha}, \tau^2 \bar{\alpha}, \dots$ that satisfies the sufficient decrease condition $f(x_t) - f(x_{t+1}) \geq r \alpha_t \| \nabla f(x_t) \|^2$---see \cite[Alg.~3.1]{nocedal2006optimization}.

The resulting algorithm can be viewed in two ways, neither of which fits neatly into existing frameworks for saddle avoidance.
First, one can write the iteration as $x_{t+1} = g(x_t)$, where $g(x) = x - \alpha(x) \nabla f(x)$ and $\alpha(x)$ is the step size returned by the line-search at point $x$.
This perspective reveals that $g$ is generally discontinuous: the discontinuity arises because small perturbations in $x$ can change the outcome of the line-search, leading to sudden jumps in the chosen step size.
As a result, the analysis tools that rely on smoothness of the iteration map $g$ do not apply.

Alternatively, one can express the update rule as $x_{t+1} = g_{i(x_t)}(x_t)$, where each 
$g_i(x) = x - \tau^i \bar{\alpha} \nabla f(x)$ is smooth, and the index $i(x_t)$ 
selects the smallest integer exponent $i$ satisfying the sufficient decrease condition.
This viewpoint retains smoothness of each candidate iteration map, but introduces a new complication: the selection $i(x_t)$ depends on the current iterate $x_t$ which itself depends on $x_0$.
Consequently, the algorithm does not apply a pre-specified sequence of iteration maps (common to all $x_0$).
As such, classical analyses of non-adaptive dynamical systems do not apply.

In both formulations, the essential difficulty is the same: the iteration lacks the regularity or structure assumed by the CSMT and related tools.
These difficulties might be resolved if, for example, the selected step sizes eventually become constant or periodic, but this does not seem to be the case with the standard line-search described above.

In order to make some progress, we focus on the second viewpoint with a slightly modified procedure.
In doing so, we aim to retain two fundamental aspects of line-search methods:
\begin{description}
	\item[Function-class adaptivity.] When $\nabla f$ is globally $L$-Lipschitz, the method achieves almost the same (worst-case) performance as if $L$ were known and we had set $\alpha_t = 1/L$.
	\item[Local-geometry adaptivity.] The step size automatically adjusts to the local curvature of the function, without requiring any explicit estimates.
\end{description}
Guided by these principles, we propose the modified procedure in Algorithm~\ref{algo:bkt-stab}.
Given an arbitrary threshold $\varepsilon > 0$, we perform a standard line-search whenever the gradient norm is at least $\varepsilon$.
If the gradient norm falls below this threshold, rather than starting the line-search from the initial step size $\bar{\alpha}$ at the next iteration,
we use the most recently selected step size $\alpha_{t}$ as the first candidate.
This is the only departure from the standard Armijo backtracking line-search scheme, and it only affects the trajectory in the small-gradient regime.
If the trajectory eventually escapes a saddle, the algorithm naturally returns to the standard line-search regime.

\begin{algorithm}[t]
	\caption{Gradient descent with a stabilized Armijo backtracking line-search}
	\begin{algorithmic}[1]
	\Statex \textbf{Input:} \(\tau, r \in (0, 1), x_0 \in \calM, \bar{\alpha} > 0, \varepsilon > 0, \) retraction $\Retr$, e.g., if $\calM = \reals^n$, $\Retr_x(v) = x+v$
	\State \(\alpha_0 \gets \bar{\alpha}\)
	\For{\(t = 0, 1, 2, \dots\)}
		\While{\(f(x_t) - f(\Retr_{x_t}(-\alpha_t \grad f(x_t))) < r \alpha_t \| \grad f(x_t) \|^2\)}
			\State \(\alpha_t \gets \tau \alpha_t\)
		\EndWhile 
		\State \(x_{t+1} \gets \Retr_{x_t} (- \alpha_t \grad f(x_t))\)
		\If{\(\|\grad f(x_{t+1})\| \ge \varepsilon\)}
			\State \(\alpha_{t+1} \gets \bar\alpha\) \Comment{Reset initial step size for standard line-search}
		\Else
			\State \(\alpha_{t+1} \gets \alpha_t\) \Comment{Maintain previously used step size instead} \label{line:non-increasing-step}
		\EndIf
	\EndFor
	\end{algorithmic}
	\label{algo:bkt-stab}
\end{algorithm}

\noindent As we shall see, the modification preserves both adaptivity notions above, while ensuring that if the algorithm converges, then the step size eventually stabilizes. 
Once the step size is constant, the update rule reduces to a smooth, fixed-step-size map.
Exploiting this perspective, our main Euclidean result is as follows.

\begin{definition} \label{def:strictsaddlesandavoidance}
	Let $f \colon \Rn \to \reals$ be $C^2$.
	A \emph{strict saddle} of $f$ is a point where the gradient of $f$ is zero and the Hessian of $f$ has a negative eigenvalue.
	An iterative minimization algorithm applied to $f$ \emph{avoids the strict saddles of $f$} if the set of initial points $x_0$ such that the algorithm converges to a strict saddle of $f$ has measure zero.
\end{definition}
\begin{theorem} \label{thm:gd-linesearch-intro}
	Let $f \colon \Rn \to \reals$ be $C^2$.
	For all $\tau, r \in (0, 1)$, all $\varepsilon > 0$ and almost any initial step size $\bar{\alpha} > 0$, the stabilized backtracking line-search gradient descent (Algorithm~\ref{algo:bkt-stab}) avoids the strict saddles of $f$.
\end{theorem}

The conclusion of Theorem~\ref{thm:gd-linesearch-intro} cannot be relaxed to non-strict saddles.
Consider $f(x) = x^3$: for all $x_0 \in [0, 1-r]$ and $\bar{\alpha} \leq 1/3$,
the algorithm selects $\alpha_t = \bar{\alpha}$ for all $t$
(reducing to GD with constant step size) and it converges to the non-strict saddle $x = 0$.
Also note that, as stated, Definition~\ref{def:strictsaddlesandavoidance} includes some local maxima among the strict saddles.
This is acceptable, as these are also critical points to which, typically, we do not expect or desire algorithms to converge.

Details of Theorem~\ref{thm:gd-linesearch-intro} (including a generalization to Riemannian manifolds) are given in Section~\ref{section:line-search}.
To get there, we first build up general results about avoidance of unstable fixed points in dynamical systems (Section~\ref{section:ingredients}), and improved guarantees for GD with constant step size (Section~\ref{section:constantstepsizeGD}).
Afterwards, we provide additional results for the Riemannian case with constant (not necessarily generic) step size (Sections~\ref{section:additional-results-c2} and~\ref{section:metric-projection-retraction}).

\subsection{Proof ideas}

At every iteration, Algorithm~\ref{algo:bkt-stab} computes the next point as $x_{t+1} = g_{i_t}(x_t)$ with some map $g_{i_t}$ selected from the collection of maps $g_i(x) = x - \tau^i \bar{\alpha} \nabla f(x)$ where $i \geq 0$ is an integer.
If the sequence of points $x_0, x_1, x_2, \ldots$ converges, then in a first phase, the step sizes can oscillate, but \emph{eventually} they become non-increasing and (we shall argue) stabilize.

Thus, the algorithm has two phases.
In the early phase, before the step size stabilizes, we run standard Armijo line-search, so the algorithm selects a sequence of iteration maps that may vary from step to step.
This sequence is not predetermined: it depends on $x_0$ in a discontinuous fashion.
In the later phase, once the accepted step size becomes constant, the algorithm reduces to gradient descent with a fixed step size.
In that regime, the tools from smooth dynamical systems such as the CSMT apply.
We handle this with the usual proof by \citet{lee2019strictsaddles}, only with a necessary improvement in the ingredients.

Accordingly, the analysis has two parts.
First, we assert local control using the CSMT.
Explicitly, we show that there exists a set $W$ of measure zero such that, if the algorithm converges to a strict saddle, then the iterates eventually all belong to $W$.
Second, we globalize the result by noting that if some iterate $x_t$ belongs to $W$, then $x_{t-1}$ belongs to the preimage of $W$ by the iteration map that was chosen at step $t-1$.
If that map has the property that preimages of measure-zero sets have measure zero---called the \emph{Luzin $N^{-1}$ property} (Definition~\ref{def:luzin})---then we see that $x_{t-1}$ also belongs to a measure zero set.
We do not know which iteration map was chosen at step $t-1$, but there are only countably many options, which is good enough.
Iterating this reasoning, we find that $x_0$ itself belongs to a measure-zero set, as desired.
Thus, the local-to-global argument only requires the individual iteration maps $g_i$ to have the Luzin $N^{-1}$ property.


The classical proof assumes that $\nabla f$ is $L$-Lipschitz continuous, and that the constant step size $\alpha$ satisfies $\alpha < 1/L$, or, in some extensions, $\alpha < 2/L$ (see \emph{related work}).
However, the main purpose of a line-search method is to avoid such  impractical restrictions on the step size.
Thus, the step sizes we entertain may be far larger than previously allowed, and that is true of the eventual (stabilized) step size as well.
Large step sizes are incompatible with the usual proof techniques for the following reasons:
\begin{itemize}
	\item The version of the CSMT used in this literature~\cite[Thm.~III.7]{shub1987book} requires the differential of the iteration map to be invertible at strict saddles.
	\item The Luzin $N^{-1}$ property has so far only been established for GD with step sizes up to $2/L$, where $L$ is the Lipschitz constant of the gradient. 
\end{itemize}

To resolve this, we call upon a better (and just as classical) version of the CSMT~\citep[Thm.~5.1]{hirsch1977invariant} which lifts the aforementioned restriction.
Then, as discussed in Section~\ref{section:ingredients}, if the differential of the iteration map is invertible \emph{almost} everywhere, the Luzin $N^{-1}$ property holds without requirements at any specific points (not even at the saddle points themselves).
This yields the following \emph{general} statement about dynamical systems, proved in Section~\ref{section:ingredients}.
Compare with~\citep[Thm.~2]{lee2019strictsaddles} which requires $\D g(x)$ to be invertible for all $x$.
\begin{theorem} \label{thm:generalsystemintro}
	Let $\calM$ be a smooth manifold.
	Consider the dynamical system $x_{t+1} = g(x_t)$ with $C^1$ iteration map $g \colon \calM \to \calM$.
	Assume $\D g(x)$ is invertible for almost all $x \in \calM$.
	Then, the set of initial points $x_0$ such that the system converges to an unstable fixed point (a point $x^*$ such that $g(x^*) = x^*$ and $\D g(x^*)$ has an eigenvalue $\lambda$ with $|\lambda| > 1$) has measure zero.
\end{theorem}

Then, using Fubini--Tonelli-type arguments, we show that the GD iteration map satisfies the Luzin $N^{-1}$ property for almost all choices of step size (Lemma~\ref{lemma:gd-luzin-almost-any-alpha}).
We later use this to show that for generic $\bar{\alpha}$, \emph{all} maps considered by the line-search have the Luzin $N^{-1}$ property.
In particular, unlike previous work, we no longer need to assume that $\nabla f$ is globally Lipschitz continuous.


\subsection{Extension to the Riemannian case}

We extend the result above to the minimization of a function $f$ on a Riemannian manifold $\calM$.
Specifically, we consider Riemannian gradient descent (RGD) as $x_{t+1} = \Retr_{x_t}(-\alpha_t \grad f(x_t))$, where both the manifold and the retraction $\Retr$ are real analytic.
Recall that a retraction is a smooth map $\Retr \colon \T\calM \to \calM \colon (x, v) \mapsto \Retr_x(v)$ on the tangent bundle of $\calM$ such that each curve of the form $c(t) = \Retr_x(tv)$ satisfies $c(0) = x$ and $c'(0) = v$.
See Sections~\ref{section:RGDgeneralsetup} and~\ref{section:analytic} for further definitions and discussion, and Section~\ref{section:line-search} for a proof of the following.

\begin{theorem} \label{thm:riemannian-line-search-intro}
	Let $f \colon \calM \to \reals$ be $C^2$ on a real-analytic manifold $\calM$ with a Riemannian metric and a real-analytic retraction $\Retr$.
	For all $\tau, r \in (0, 1)$, all $\varepsilon > 0$ and almost any initial step size $\bar{\alpha} > 0$, the stabilized backtracking line-search Riemannian gradient descent (Algorithm~\ref{algo:bkt-stab}) avoids the strict saddles of $f$.
\end{theorem}

Of course, this encompasses the case $\calM = \Rn$ with retraction $\Retr_x(v) = x + v$, as in Theorem~\ref{thm:gd-linesearch-intro}.

In Theorem~\ref{thm:riemannian-line-search-intro}, some condition on the retraction is indeed necessary (though it remains open whether real-analyticity is the best choice).
This is because retractions need only satisfy local conditions near the origins of the tangent spaces.
Away from this neighborhood, they are a priori uncontrolled, allowing for counter-examples---see Remark~\ref{remark:rgd-analytic-requires-retraction-condition}.

Notwithstanding, the following remarks highlight that the real-analyticity assumptions on the manifold and retraction are satisfied in many practical situations.
(Note that $f$ only needs to be $C^2$.)

\begin{remark} \label{remark:real-analytic-manifolds}
	Common examples of real-analytic manifolds include the unit sphere $\Sn$ in $\Rn$, 
	the Stiefel manifold $\St(n, p)$ of orthonormal $p$-frames in $\Rn$,
	the orthogonal group $\mathrm{O}(n)$,
	the manifold of real matrices of fixed rank $r$ and
	the set of symmetric positive definite matrices,
	as well as products thereof.
\end{remark}

\begin{remark} \label{remark:exp-real-analytic}
	On a real-analytic manifold, if the Riemannian metric is also real analytic (which is common though not required by Theorem~\ref{thm:riemannian-line-search-intro}),
	then the exponential map (see Section~\ref{section:additional-results-c2}) is real analytic everywhere on its domain.
	This is because it emerges from solving real-analytic differential equations, which have real-analytic solutions by the Cauchy--Kowalewski theorem~\citep[\S2.8]{krantz2002primer}.

	The exponential map is a retraction with possibly restricted domain.
	For simplicity, we require the retraction to be defined on the entire tangent bundle.
	This is the case for the exponential map if $\calM$ is \emph{complete}~\citep[Thm.~6.19]{lee2018riemannian}.
	Thus, for $\Sn$, $\St(n, p)$ and $\mathrm{O}(n)$ with their usual metrics (and products thereof), the exponential map is a real-analytic retraction defined on the entire tangent bundle $\T\calM$.
\end{remark}	

\begin{remark} \label{remark:metric-projection-real-analytic}
	The metric projection onto a real-analytic submanifold $\calM$ of a Euclidean space $\calE$ is also real analytic wherever it is uniquely defined.
	Indeed, let $\Omega \subseteq \calE$ be the interior of the set of points $y \in \calE$ whose metric projection $\calP_\calM(y) := \argmin{x \in \calM} \|x - y\|$ is a singleton.
	The restriction $\calP_{\calM}\vert_{\Omega}$ is real analytic \citep[Thm.~4.1]{dudek1994nonlinearprojection},
	so the \emph{metric projection retraction} defined by $\Retr_x(v) = \calP_\calM(x + v)$ is real analytic on its domain.
	For $\Sn$, $\St(n, p)$ and $\mathrm{O}(n)$ (and products thereof), this domain is all of $\T\calM$~\citep[\S5.12]{boumal2020intromanifolds}.
\end{remark}

\subsection{Further results for RGD with a specific constant step size}

In the Euclidean setting, we show that GD with a \emph{fixed} step size $\alpha$ avoids strict saddles for \emph{almost all} $\alpha > 0$.
This complements classical results which---assuming $\nabla f$ is $L$-Lipschitz---provide the same guarantee for \emph{all} $\alpha$ in the interval $(0, 1/L)$.
Both perspectives are valuable.

In the Riemannian setting, we provide the same ``almost all'' guarantee when assuming real-analyticity.
However, the corresponding classical results are not nearly as satisfying.
For example, \citet{lee2019strictsaddles} provide an ``interval'' guarantee for compact manifolds with the metric projection retraction, but the interval may be rather small.

To complement these results, we improve upon the latter by showing that, without assuming real-analyticity, the strict saddle avoidance guarantee holds for RGD with all step sizes in some interval, and the interval may be as large as $(0, 1/L)$ in some cases (e.g., for Hadamard manifolds with the exponential retraction).
See Sections~\ref{section:additional-results-c2} and~\ref{section:metric-projection-retraction}.

\subsection{Related work}

\paragraph{Asymptotic saddle avoidance.}
As outlined in the introduction, the most influential papers on the topic of saddle avoidance for optimization algorithms are those by \citet{lee2016gradient} and \citet{panageas2017gradient}, and their joint paper~\citep{lee2019strictsaddles}.

Prior to those, the idea of using the stable manifold theorem to argue generic convergence to a local minimum already appears in the optimization literature with work by \citet[Appendix]{goudou2009gradient}.
They studied a continuous-time flow inspired by the heavy ball method applied to a coercive Morse function.

The base result by \citet{lee2019strictsaddles} regards GD with constant step size $\alpha < 1/L$, where $\nabla f$ is $L$-Lipschitz continuous.
They also obtain results for other optimization algorithms including coordinate descent, block coordinate descent, mirror descent and Riemannian gradient descent with metric projection retraction.
\citet{li2019alternating} apply similar ideas to study alternating minimization and proximal alternating minimization. 

\citet{schaeffer2020extending} extend the step size restriction on GD to allow $\alpha < \frac{2}{L}$, provided that the set $\{ x \in \Rn \mid \textrm{some eigenvalue of } \nabla^2 f(x) \textrm{ equals } 1/\alpha \}$ has measure zero and does not contain any saddle.
This assumption is a priori non-trivial to check: see Remark~\ref{rem:schaefferassumption} for additional context.

The Riemannian case is also analyzed by \citet{hou2020analysis}.
They argue that if $\alpha$ is small enough, then $g$ is a local diffeomorphism around critical points of $f$ (as in Lemma~\ref{lem:Dg-not-singular-critical-point}).
However, this is not sufficient to ensure that $g^{-1}$ preserves sets of measure zero globally, which hampers the local-to-global part of their argument.
See Section~\ref{sec:JacobiDgx} for the additional technical tools that enable the study of $\D g(x)$ at non-critical points.

\paragraph{Saddle avoidance in stochastic optimization.}
Classical results on saddle avoidance in stochastic methods go back to \citet{pemantle1990nonconvergence}, who showed that stochastic gradient methods (SGD) avoid hyperbolic unstable fixed points.
More recently, \citet{mertikopoulos2020almost} proved that SGD avoids strict saddle manifolds, with extensions to Riemannian settings by \citet{hsieh2024riemannian}.
While these results are asymptotic, non-asymptotic analyses show that GD may take linear time to escape a strict saddle if initialized with sufficient unstable component \citep{dixit2022boundary}, but can require exponentially many steps when initialized in certain regions of non-zero measure near the saddle \cite[\S3]{du2017gradient}.
In contrast, stochastic methods with isotropic noise can escape saddles and reach approximate local minima in polynomial time with high probability \citep{ge2015escaping, jin2017howtoescape, jin2018agdescapes}, including in the Riemannian setting via noise injected on the manifold or tangent space \citep{sun2019prgd, criscitiello2019escapingsaddles}.

\paragraph{Non-constant step sizes.} 
Several works study saddle avoidance for GD with non-constant step sizes, assuming an \textit{a priori} fixed sequence that is independent of the initial point $x_0$.
\citet{panageas2019first} show that GD (among others) with step sizes $\alpha_t$ decaying as $\Omega(1/t)$ avoids \emph{isolated} strict saddle points. 
\citet[Cor.~3.8]{schaeffer2020extending} consider GD with a pre-set sequence of piecewise constant step sizes. 
More specifically, the step size $\alpha_t$ at time $t$ is a piecewise constant function of $t$ with finitely many jumps.
Additionally, the result requires that $\nabla f$ is globally $L$-Lipschitz continuous and that $\alpha_t L \in (0, 2)$ for all times $t \geq 0$, which does not hold in practice for line-search methods.

Another result is presented by \citet[Thm.~7.2]{shi2020mathematical}, who show that GD avoids strict saddles when using a pre-set sequence of step sizes, assuming a globally Lipschitz Hessian and that all strict saddles are isolated.
The step sizes must be fixed in advance, which prevents the method from adapting to the global problem class.
In addition, even though they are predetermined, the step sizes must satisfy conditions that depend on the trajectory itself, specifically on the local Lipschitz constant of the gradient around the current iterate, which is not known a priori.
Therefore, the method is not adaptive to function class or local geometry.

There have been other attempts at establishing saddle avoidance for GD with line-search-type methods.
For example, \citet{truong2025some} proposes a line-search that requires an explicit approximation of the local Lipschitz constant of the gradient at each step, which is not easily implementable.

\citet{dixit2026accelerated} establish strict saddle avoidance for Nesterov-type accelerated methods with varying momentum parameters.

\paragraph{Non-smooth settings.}
\citet{davis2022proximal} and \citet{davis2025active} study saddle avoidance properties of subgradient and proximal point methods for a non-smooth cost function $f$.
This notably requires redefining the notion of strict saddle.
(In fact, this is necessary even if $f$ is $C^1$, as then $f$ may not have a Hessian.)
One option is to consider critical points with a direction along which $f$ decreases quadratically.
However, this would be inadequate: there exist $C^1$ functions with a set $S$ of positive measure such that GD initialized in $S$ converges to such a point \cite[\S1.1]{davis2022proximal}.
In part for this reason, \citet{davis2022proximal} propose the concept of ``active strict saddle''.

\section{Unstable fixed point avoidance for dynamical systems} \label{section:ingredients}

En route to studying gradient descent, we first study general dynamical systems of the form $x_{t+1} = g(x_t)$ on a smooth manifold $\calM$.
One of the main products of this section is a proof of Theorem~\ref{thm:generalsystemintro}.
Let us provide a few definitions to clarify its statement.
\begin{definition}
	Let $g \colon \calM \to \calM$ be $C^1$.
	A point $x^\ast \in \calM$ is a \emph{fixed point} if $g(x^\ast) = x^\ast$.
	It is an \emph{unstable fixed point} if also $\D g(x^\ast)$ has at least one eigenvalue $\lambda$ with magnitude $|\lambda| > 1$.
\end{definition}
(Such a point might also be called a \emph{strictly} unstable fixed point.)

In the context of this paper, the strict saddles of a cost function $f$ would be unstable fixed points of an optimization algorithm for $f$ with iteration map $g$.

The Lebesgue measure of a set $A \subseteq \Rn$ is denoted by $\mu_n(A)$ or simply $\mu(A)$.
We say $A$ has \emph{measure zero} if $\mu(A) = 0$.
More generally, on manifolds, we use the following standard notion of measure zero sets~\cite[Ch.~6]{lee2012smoothmanifolds}.

\begin{definition}
	Let $\calM$ be a smooth manifold of dimension $n$.
	A subset $A \subseteq \calM$ has \emph{measure zero} in $\calM$ if, for every smooth chart 
	$(U, \varphi)$ of $\calM$, the subset $\varphi(A \cap U) \subseteq \Rn$ has Lebesgue measure zero.
	With some abuse of notation, we then write $\mu(A) = 0$.
\end{definition}
Conveniently, it suffices to check the condition above for any collection of charts whose domains cover $A$~\cite[Lem.~6.6]{lee2012smoothmanifolds}: the conclusion is independent of the choice of charts.

A dynamical system \emph{avoids unstable fixed points} if the set of initialization points $x_0$ that lead to convergence to an unstable fixed point has measure zero.
This is made explicit below with $S$ denoting the set of unstable fixed points.

\begin{definition} \label{def:systemavoids}
	The \emph{stable set} of a point $x^\ast \in \calM$ for $g \colon \calM \to \calM$ is the set
	\begin{equation*}
		W(x^\ast) = 
		\left\{ x_0 \in \calM \mid \lim_{t \to \infty} g^t(x_0) = x^\ast \right\}.
	\end{equation*}
	We say $g$ \emph{avoids} $x^\ast$ if $W(x^\ast)$ has measure zero.
	Moreover, $g$ \emph{avoids a subset} $S \subseteq \calM$ if $W(S) = \bigcup_{x^\ast \in S} W(x^\ast)$ has measure zero.
\end{definition}

\begin{remark}
	Notice that a system might still accumulate at a point $x^*$ and not converge there, and this will not be captured by the notion of ``avoidance''.
	However, such behavior is rather uncommon in optimization, see Remark~\ref{remark:no-limit-system} and Theorem~\ref{thm:soc-convergence}.
\end{remark}

\subsection{A general unstable fixed point avoidance theorem}

To establish a global unstable fixed point avoidance result, we require the iteration map $g$ not to ``collapse'' sets of positive measure to sets of measure zero.
Under that assumption, repeated preimages of a measure zero set also have measure zero, enabling the local-to-global part of the proof.
This condition appears in the literature as the \emph{Luzin $N^{-1}$ property} (also spelled \emph{Lusin})
\citep[Def.~4.12]{hencl2014lectures} or the \emph{0-property} \citep{ponomarev1987submersions}.

\begin{definition} \label{def:luzin}
	A continuous map $g \colon \calM \rightarrow \calM$ has the \emph{Luzin $N^{-1}$ property} if 
	\begin{align*}
		\textrm{for all } E \subseteq \calM, \qquad
		\mu(E) = 0
		\quad \implies \quad
		\mu(g^{-1}(E)) = 0.
		\tag{Luzin $N^{-1}$}
	\end{align*}
\end{definition}

With this condition, we can state the template theorem for unstable fixed point avoidance.
Subsequent saddle avoidance results rely on this theorem or on a variation of its proof. 

\begin{theorem} \label{thm:saddle-avoidance-template}
	If $g \colon \calM \to \calM$ is $C^1$ near its fixed points and it has the Luzin $N^{-1}$ property, then $g$ avoids its unstable fixed points.
\end{theorem}

A few remarks are in order before we proceed to the proof.

\begin{remark}
	The $C^1$ assumption in Theorem~\ref{thm:saddle-avoidance-template} is only used for the local part of the argument, based on the CSMT, while the Luzin $N^{-1}$ property is used for obtaining a global result.
\end{remark}

\begin{remark} \label{remark:luzin-not-removable}
	The Luzin $N^{-1}$ condition in Theorem~\ref{thm:saddle-avoidance-template} cannot be omitted.
	We illustrate this with a counterexample of the form $g(x) = x - \alpha \nabla f(x)$, where $\alpha = 1$ and $f \colon \reals^2 \rightarrow \reals$ is designed below.

	The idea is for $f$ to behave like a squared norm far from the origin, while arranging for the origin itself to be a strict saddle (see Figure~\ref{fig:landscape-contour}).
	Explicitly,
	\begin{itemize}
		\item For $\|x\| < 1$, we set $f(x) = f_1(x) := (x_1^2 - x_2^2) / 2$, 
		so that the origin is a strict saddle.
		\item For $\|x\| > 2$, we set $f(x) = f_2(x) := (x_1^2 + x_2^2) / 2$. 
	\end{itemize}
	To smoothly interpolate between $f_1$ and $f_2$, we use a smooth transition function $q \colon \reals \to \reals$ defined by
	\begin{equation*}
		q(t) = \frac{u(t)}{u(t) + u(3-t)}, \qquad \textrm{where} \qquad u(t) = \begin{cases}
		e^{-3/t} & \text{if } t > 0, \\
		0 & \text{otherwise}.
		\end{cases}
	\end{equation*}
	It satisfies $q(t) = 0$ for $t \leq 0$ and $q(t) = 1$ for $t \geq 3$.
	With this, let $f(x)$ be defined for all $x$ as
	\begin{equation*}
		f(x) = q(4 - \|x\|^2) \; f_1(x) + \big(1 - q(4 - \|x\|^2) \big) \; f_2(x).
	\end{equation*}
	Notice that $f$ is smooth and satisfies the requirements above.
	Now, consider the map $g(x) = x - \nabla f(x)$.
	The origin is an unstable fixed point for $g$ since $g(0) = 0$ and $\D g(0)$ has eigenvalues $0$ and $2$.
	The origin is also a strict saddle for $f$.
	Moreover, for all $x \in \reals^2$ with $\|x\| > 2$, we have $g(x) = 0$ so that $g$ does not avoid its unstable fixed points, and indeed it does not have the Luzin $N^{-1}$ property.
\end{remark}

\begin{figure}[htbp]
    \centering

    \begin{subfigure}{0.48\textwidth}
        \centering
        \includegraphics[width=1.1\textwidth]{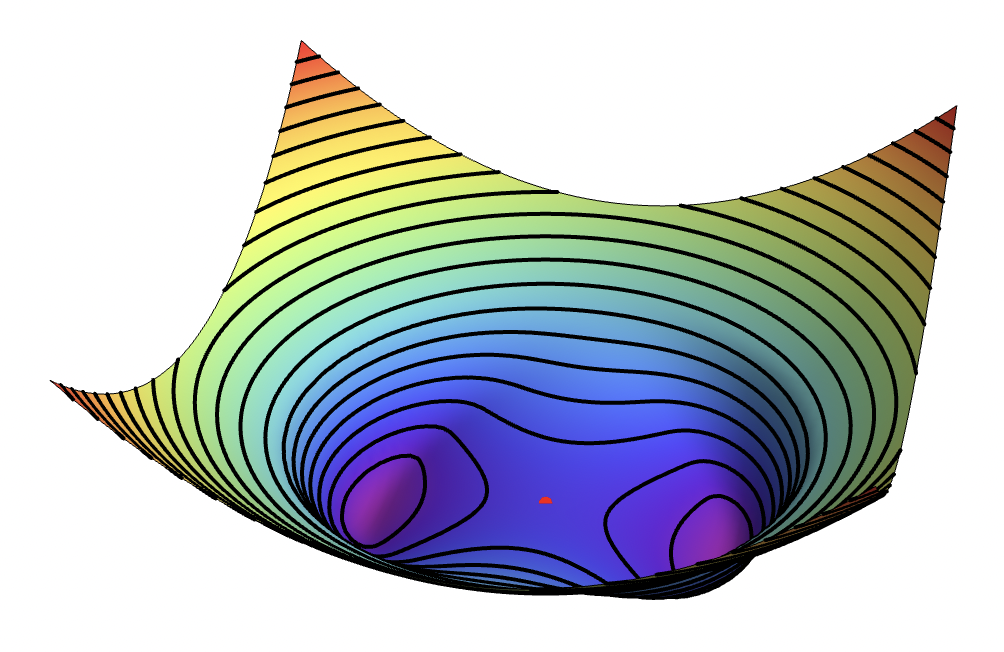}
        \caption{Landscape plot}
        \label{fig:landscape}
    \end{subfigure}
    \hfill
    \begin{subfigure}{0.48\textwidth}
        \centering
        \includegraphics[width=0.7\textwidth]{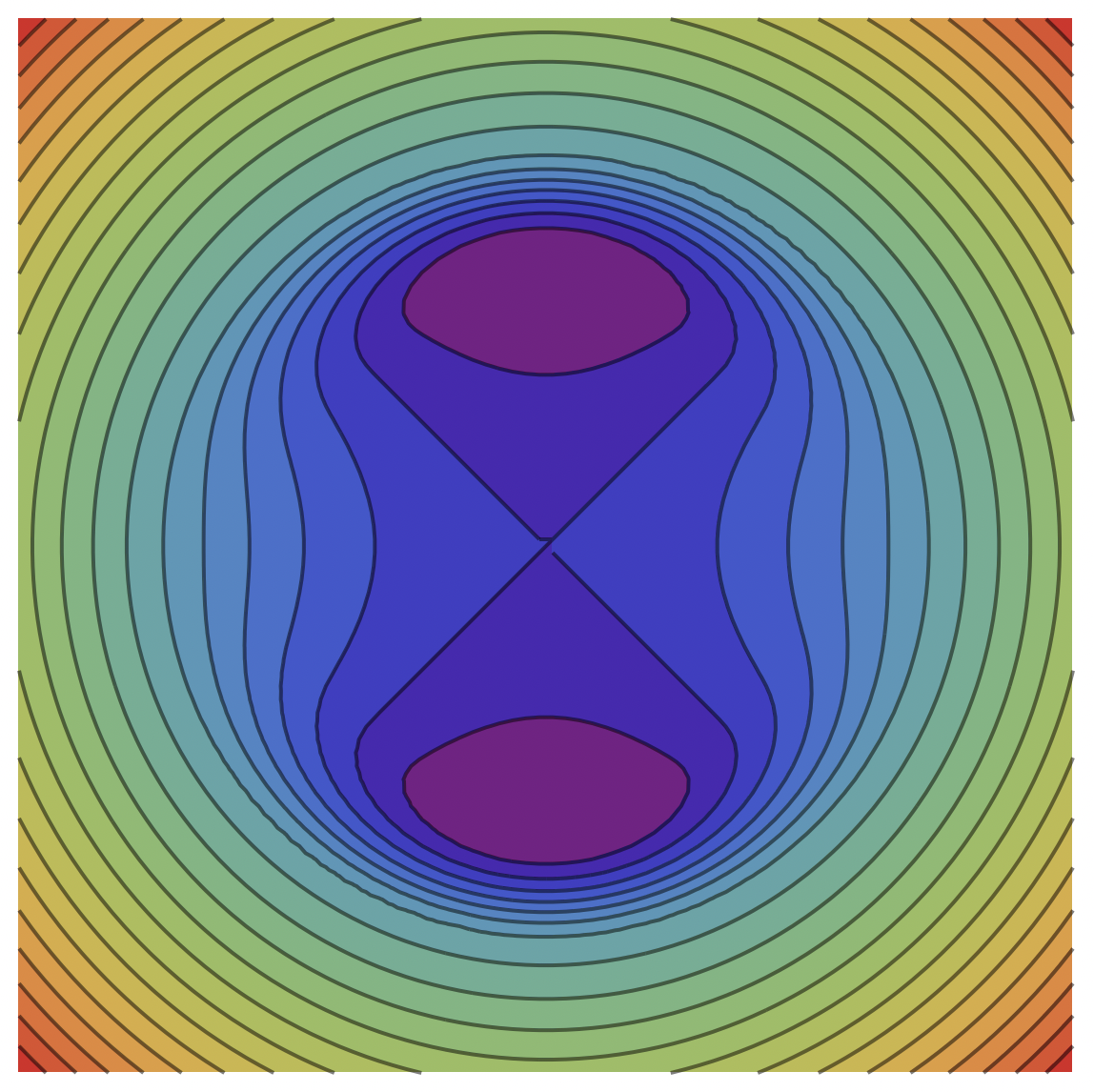}
        \caption{Contour plot}
        \label{fig:contour}
    \end{subfigure}

    \caption{Counterexample function $f$ from Remark~\ref{remark:luzin-not-removable}.
	Here, $g$ is the gradient descent map with step size $\alpha=1$.
	The origin is an unstable fixed point, so locally, $g$ avoids it.
	However, $g$ does not satisfy the Luzin $N^{-1}$ property, and the region ${\|x\| > 2}$ is mapped to the origin in one step.
	}
    \label{fig:landscape-contour}
\end{figure}

\begin{remark} \label{rem:Luzinisright}
	Theorem~\ref{thm:saddle-avoidance-template} generalizes existing fixed point avoidance
	results from the literature.
	For example, \citet{panageas2017gradient} require the map $g$ to be a global diffeomorphism,
	while \citet{lee2019strictsaddles} relax this assumption by requiring 
	$\D g(x)$ to be invertible for all $x$ in the domain.
	Both conditions are more restrictive than the Luzin $N^{-1}$ condition, as demonstrated below in Theorem~\ref{thm:luzin-property}.
	Still, the proof in Section~\ref{subsection:proof-general-thm} is the same as theirs: only the ingredients change.
\end{remark}

\begin{remark}
	One may also consider a continuous dynamical system $x'(t) = V(x(t))$ with a smooth vector field $V$ on $\calM$ such that the flow $\phi \colon \reals \times \calM \to \calM$ is defined for all initial conditions $x(0) = x_0$ and all times.
	The time-one map $\phi_1 \colon \calM \to \calM$ flows a given point for one unit of time, thus mapping $x(0)$ to $x(1)$.
	Since $\phi_1$ is a diffeomorphism, it has the Luzin~$N^{-1}$ property~\cite[Thm.~6.9, Thm.~9.12]{lee2012smoothmanifolds}.
	Then, we obtain an avoidance result for the continuous system from Theorem~\ref{thm:saddle-avoidance-template} with $g := \phi_1$.
	For the case of gradient flows, a similar result is established in~\cite[Lem.~A.1]{geshkovski2025mathtransformers}. 
\end{remark}

\subsection{Proof of the general theorem} \label{subsection:proof-general-thm}

As recalled in the introduction, the now standard proof of global avoidance of unstable fixed points relies on the Center Stable Manifold Theorem (CSMT).
The latter provides a local characterization of the set of initial conditions from which 
repeated iteration of $g$ converges to a fixed point.
Importantly, it ensures that if the iteration map has at least one expanding direction, 
then this set has measure zero.

The CSMT commonly used in the saddle avoidance literature is the one stated in \cite[Thm.~III.7]{shub1987book}.
It requires that the iteration map is a diffeomorphism in a neighborhood of the fixed point where it is applied.
However, using a different proof technique, the result still holds even if $g$ is not locally invertible: see \citep[Thm.~5.1]{hirsch1977invariant}.
Theorem~\ref{thm:csmt-no-local-diffeomorphism} below states the parts we need from this more general version of the CSMT.
(\citet[Ex.~III.2]{shub1987book} also includes it as an exercise.)

This generalized form makes it possible to treat cases where the iteration map is not necessarily a diffeomorphism, nor even a local diffeomorphism around fixed points.
Crucially, in the analysis of saddle avoidance for gradient descent, this enables us to remove previously required assumptions on the step size.
The ensuing analysis is also simpler.

\begin{theorem}[{\protect{\citet[Thm.~5.1]{hirsch1977invariant}}}]
	Let $x^\ast$ be a fixed point of the $C^1$ map $g \colon \calM \rightarrow \calM$. 
	Assume $\D g(x^\ast)$ has at least one eigenvalue $\lambda$ with magnitude $|\lambda| > 1$.
	Then, there exist 
	\begin{itemize}
		\item an open neighborhood $B$ of $x^\ast$, and
		\item a measure zero set $\Wcsloc$ in $\calM$
	\end{itemize}
	such that if $\{ x_t \}_{t \geq 0}$ is a sequence generated by $x_{t+1} = g(x_t)$ and $x_t$ is in $B$ for all $t \geq 0$, then $x_t$ is also in $\Wcsloc$ for all $t \geq 0$.
\label{thm:csmt-no-local-diffeomorphism}
\end{theorem}

It is worth noting that the full CSMT provides additional information we do not need, such as the regularity of $\Wcsloc$ near $x^\ast$.
Moreover, it should be emphasized that the theorem is informative even for trajectories that converge to a point other than $x^\ast$ yet within the neighborhood $B$.
This is particularly useful in cases where unstable fixed points are non-isolated, as first noted by \citet{panageas2017gradient}.

\begin{proof}[Proof of Theorem~\ref{thm:saddle-avoidance-template}]
The proof is the same as in \citep[Thm.~2]{lee2019strictsaddles}:
the (better) CSMT and Lindel\"{o}f's lemma are the central ingredients.

Assume that $x^\ast$ is an unstable fixed point of $g$.
Applying Theorem~\ref{thm:csmt-no-local-diffeomorphism} at $x^\ast$, we obtain a neighborhood $B_{x^\ast}$ of $x^\ast$ and a local center stable set $\Wcsloc(x^\ast)$ of measure zero such that
\begin{equation} \label{eq:csmt-conclusion}
	g^t(x) \in B_{x^\ast} \textrm{ for all } t \geq 0 \quad \implies \quad x \in \Wcsloc(x^\ast).
\end{equation}
Let $S$ denote the set of all unstable fixed points of $g$.
Apply Theorem~\ref{thm:csmt-no-local-diffeomorphism} at each $x^\ast \in S$.
The sets $ \{ B_{x^\ast} \}_{x^\ast \in S}$ form a (potentially uncountable) cover for $S$. 
Since $\calM$ is second countable, Lindel\"{o}f's lemma allows us to extract a countable subcover of $S$.
That is, there exists a countable subset $\{ x_i^\ast \}_{i \geq 0}$ 
of $S$ such that $\{ B_{x_i^\ast} \}_{i \geq 0}$ covers $S$.

Now, consider $x_0 \in \calM$ such that $\lim_{t \to \infty} g^t(x_0) = x^\ast \in S$
and choose one of the balls $B_{x_i^\ast}$ in the countable subcover such that $x^\ast \in B_{x_i^\ast}$.
Then, since the iterates converge to $x^\ast$, they must eventually enter $B_{x_i^\ast}$ and never 
exit it again. 
Let $m \geq 0$ be the index when this happens.
That is, we have that $g^t(x_0)$ is in $B_{x_i^\ast}$ for all $t \geq m$. 
From~\eqref{eq:csmt-conclusion}, we deduce that $g^m(x_0)$ is in $\Wcsloc(x_i^\ast)$, and so $x_0$ is in $g^{-m}(\Wcsloc(x_i^\ast))$. 
Therefore, the set of all initial points $x_0$ which could lead to convergence to an unstable fixed point is included in the set 
\begin{equation*}
	W = \bigcup_{i \geq 0} \bigcup_{m \geq 0} g^{-m} (\Wcsloc(x_i^\ast)).
\end{equation*}
Since $\Wcsloc(x_i^\ast)$ has measure zero and since $g$ satisfies the Luzin 
$N^{-1}$ property, each set $g^{-m} (\Wcsloc(x_i^\ast))$ has measure zero.
Thus, $W$ has measure zero as the union of countably many measure zero sets.
\end{proof}

\subsection{Verifying the Luzin $N^{-1}$ property}

As stated in Definition~\ref{def:luzin}, the Luzin $N^{-1}$ property might seem hard to check.
Fortunately, when $g$ is continuously differentiable, the following theorem by \citet{ponomarev1987submersions} provides an alternative characterization that may be more convenient.

\begin{theorem}[\citet{ponomarev1987submersions}] \label{thm:luzin-property}
	If $g \colon \calM \rightarrow \calM$ is $C^1$, then the following are equivalent:
	\begin{itemize}
		\item $g$ has the Luzin $N^{-1}$ property;
		\item $\D g(x)$ is invertible for almost all $x \in \calM$.
	\end{itemize}
\end{theorem}
For completeness, we provide a proof of Theorem~\ref{thm:luzin-property} in Appendix~\ref{appendix:proofs-luzin}.
The proof is adapted from~\citep{ponomarev1987submersions} and relies on Lindel\"{o}f's lemma and Sard's theorem as the key technical ingredients.
Notice that Theorem~\ref{thm:generalsystemintro} is a corollary of Theorems~\ref{thm:saddle-avoidance-template} and~\ref{thm:luzin-property}.

\section{Improved guarantees for a constant step size}
\label{section:constantstepsizeGD}

Recall the notions of strict saddle and strict saddle avoidance from the introduction (Definition~\ref{def:strictsaddlesandavoidance}).
In this section, we use Theorem~\ref{thm:saddle-avoidance-template} to obtain new strict saddle avoidance results for gradient descent with constant step size.

\subsection{The Euclidean case}

For a cost function $f \colon \Rn \to \reals$, the iteration map of
gradient descent (GD) with constant step size $\alpha$ is $g_\alpha(x) = x - \alpha \nabla f(x)$.
The main result of this section is stated below.

\begin{theorem} \label{thm:euclidean-gd-result}
	Let $f \colon \Rn \to \reals$ be $C^2$.
	For almost all $\alpha > 0$, gradient descent with step size $\alpha$ avoids the strict saddles of $f$.
\end{theorem}

It is straightforward to see that strict saddles are unstable, regardless of the step size used.
Indeed, the differential of the iteration map $g_\alpha$ is $\D g_\alpha(x) = I_n - \alpha \nabla^2 f(x)$ for all $x$.
At a strict saddle, $\nabla^2 f(x)$ has at least one negative eigenvalue, so for any step size $\alpha > 0$, $\D g_\alpha(x)$ has at least one eigenvalue larger than $1$.

To apply Theorem~\ref{thm:saddle-avoidance-template}, we need to show that the iteration map satisfies the Luzin $N^{-1}$ property.
We make a few remarks, then dedicate the rest of this section to showing that, in the Euclidean case, this is indeed true for almost any step size, with no additional assumptions.

\begin{remark}
	Notice that Theorem~\ref{thm:euclidean-gd-result} applies even if the function value is not monotonically decreasing along the trajectory.
\end{remark}

\begin{remark} \label{rem:schaefferassumption}
	\citet{schaeffer2020extending} extended the usual step size condition from $\alpha < 1/L$ to 
	$\alpha < 2/L$.
	However, to do so, they introduced the assumption that the set 
	\begin{equation*}
		\Sigma_g = \bigcup_{i \in \{1, \dots, n\} } \{ x \in \Rn \mid 1 - \alpha \lambda_i(\nabla^2 f(x)) = 0 \}
	\end{equation*}
	has measure zero in $\Rn$ and does not contain strict saddles, with $\lambda_i(A)$ denoting the $i$th eigenvalue of $A$.
	The assumption that $\Sigma_g$ has measure zero implies the Luzin $N^{-1}$ condition.
	However, as stated, it is not immediately clear how to verify the assumption in practice.
	Additionally, their result does not apply if $\D g(x^\ast)$ is singular at some strict saddle.
	Instead, we show that for a $C^2$ cost and almost any step size, the corresponding iteration map 
	automatically satisfies the Luzin $N^{-1}$ property (Lemma~\ref{lemma:gd-luzin-almost-any-alpha}). 
	Additionally, we show that the condition that $\Sigma_g$ does not contain any strict saddle
	is unnecessary owing to the stronger version of the CSMT stated in Theorem~\ref{thm:csmt-no-local-diffeomorphism}.
\end{remark}

\begin{remark} \label{remark:no-limit-system}
	If $\lim_{t \to \infty} g^t(x_0)$ does not exist, then $x_0$ is not included in the stable set of any fixed point of $g$.
	In general, to ensure that the iterates of an optimization algorithm converge (to some critical
	point of the objective), both the function and the optimization algorithm must satisfy additional 
	assumptions \citep{absil2005iteratesanalytic}.
	There, the Lipschitz condition on $\nabla f$ and associated step size restriction can help.
	See also Theorem~\ref{thm:soc-convergence}.
\end{remark}

\begin{remark} \label{remark:euclidean-gd-counterexample}
	Theorem~\ref{thm:euclidean-gd-result} holds for \emph{almost all} positive $\alpha$, but indeed not for all.
	Consider again the function $f$ from Remark~\ref{remark:luzin-not-removable}.
	If $\alpha = 1$, then for all points $x \in \reals^2$ with $\|x\| > 2$, 
	we have $g(x) = 0$.
	In this case, the iteration map $g$ maps a set of positive measure
	directly to a strict saddle, showing that GD with step size $\alpha = 1$ indeed does not avoid the strict saddles of $f$.
	Theorem~\ref{thm:euclidean-gd-result} implies that the issue can be resolved by	perturbing $\alpha$.
\end{remark}

The main idea of the proof is to analyze the set $S$ of tuples $(x, \alpha)$ for which the differential
of the corresponding iteration map $\D g_\alpha(x)$ is singular.
Using a Fubini--Tonelli-type argument that we make precise in Lemma~\ref{lemma:vertical-horizontal-slices},
we show that $S$ has measure zero if and only if almost all of its slices---obtained by fixing either $x$ or 
$\alpha$--- have measure zero.
In the case of gradient descent, we then show in Lemma~\ref{lemma:gd-luzin-almost-any-alpha} that for
fixed $x$, the set of $\alpha$ that make $\D g_\alpha(x)$ singular is in fact finite.
As a consequence, for almost any fixed $\alpha$, the set of $x$ that make $\D g_\alpha(x)$ singular 
has measure zero.

\begin{lemma} \label{lemma:vertical-horizontal-slices}
	Let $S$ be a measurable subset in $\Rn \times \Rm$.
	The following statements are equivalent:
	\begin{enumerate}
		\item $S$ has measure zero in $\Rn \times \Rm$. 
		\item For almost all $x \in \Rn$, the slice $S^x = \{y \in \Rm \mid (x, y) \in S \}$
		has measure zero in $\Rm$. 
		\item For almost all $y \in \Rm$, the slice $S_y = \{ x \in \Rn \mid (x, y) \in S \}$
		has measure zero in $\reals^{n}$.
	\end{enumerate}
\end{lemma}

For completeness, we include a proof in Appendix~\ref{appendix:proofs-measure-zero}.
With this in hand, we now argue that the iteration map of GD satisfies the Luzin $N^{-1}$
property for almost any step size.

\begin{lemma} \label{lemma:gd-luzin-almost-any-alpha}
	Let $f \colon \Rn \to \reals$ be $C^2$.
	For almost all $\alpha \in \reals$, the iteration map $g_\alpha(x) = x - \alpha \nabla f(x)$ 
	satisfies the Luzin $N^{-1}$ property.
\end{lemma}
\begin{proof}
	Define the set of points $x$ and step sizes $\alpha$ such that $\D g_\alpha(x)$ is singular, that is,
	\begin{equation}
		S = \{ (x, \alpha) \in \Rn \times \reals \mid \det \D g_{\alpha}(x) = 0 \}.
	\end{equation}
	Since $S$ is a level set of a continuous function, it is closed; so it is measurable.
	Let us argue that the vertical slices $S^x$ of $S$ have measure zero in $\reals$.
	To this end, observe that
	\begin{equation*}
		\alpha \mapsto \det \D g_{\alpha}(x) = \det (I_n - \alpha \nabla^2 f(x))
	\end{equation*}
	is a polynomial in $\alpha$ of degree at most $n$.
	Also, it is not identically zero since $\det \D g_0(x) = 1$.
	Therefore, the set 
	\begin{equation*}
		S^x = \{ \alpha \in \reals \mid \det \D g_\alpha(x) = 0 \}
	\end{equation*}
	is finite, and in particular $\mu_1(S^x) = 0$ for every $x$.

	By Lemma~\ref{lemma:vertical-horizontal-slices}, we obtain that for almost all $\alpha \in \reals$,
	the slices of $S$ in the other direction also have measure zero, that is, 
	\begin{equation}
		\mu_n (S_\alpha) = 
		\mu_n \Big( \{ x \in \Rn \mid \D g_\alpha(x) \text{ is singular} \}\Big) = 0.
	\end{equation}
	Using Theorem~\ref{thm:luzin-property}, it follows that, for almost all $\alpha$, the iteration map has
	the Luzin $N^{-1}$ property, as announced.
\end{proof}

\subsection{The Riemannian case: general setup} \label{section:RGDgeneralsetup}

Let $f \colon \calM \to \reals$ be a $C^2$ cost function to be minimized on a Riemannian manifold $\calM$ equipped with a retraction $\Retr$.
Denote by $\grad f(x)$ the Riemannian gradient of $f$ at $x$.
Then, for a step size $\alpha > 0$, the Riemannian gradient descent (RGD) iteration map is given by
\begin{align}
	g \colon \calM \to \calM, \qquad
	g(x) = \Retr_x(-\alpha \grad f(x)).
	\tag{RGD}\label{eq:RGD}
\end{align}
In this section, we verify that the strict saddle points of $f$ are unstable fixed points of $g$.
Strict saddle avoidance then follows as a corollary of Theorem~\ref{thm:saddle-avoidance-template} upon verifying the Luzin $N^{-1}$ property for $g$.
We do so under various assumptions in Sections~\ref{section:analytic},~\ref{section:additional-results-c2} and~\ref{section:metric-projection-retraction}.

A critical point $x^*$ of $f$ satisfies $\grad f(x^*) = 0$, and is therefore a fixed point for $g$.
It is a strict saddle if the Riemannian Hessian $\Hess f(x^*)$ has at least one negative eigenvalue.
To show that $x^*$ is then an unstable fixed point of $g$, we must show $\D g(x^*)$ has at least one eigenvalue with magnitude larger than $1$.
This follows easily after computing the differential of $g$ at critical points, as done now.
The expression is analogous to the Euclidean one.
It also appears in \cite[Prop.~8]{lee2019strictsaddles} and \cite[Thm.~2.9]{hou2020analysis},
though restricted to the metric projection retraction.

\begin{lemma} \label{lemma:Dgx-critical}
	If $f \colon \calM \to \reals$ is $C^2$ and $x^\ast$ is a critical point of $f$, then $\D g(x^\ast) = I - \alpha \Hess f(x^\ast)$.
	In particular, strict saddles of $f$ are unstable fixed points of $g$~\eqref{eq:RGD} for all $\alpha > 0$.
\end{lemma}
\begin{proof}
	The retraction $\Retr \colon \T\calM \to \calM \colon (x, u) \mapsto \Retr(x, u) = \Retr_x(u)$ is arbitrary.
	Given $(x, u) \in \T\calM$, the chain rule provides 
	\begin{equation*} 
		\D g(x)[u] = 
		\D \Retr(x, - \alpha \grad f(x)) [(u, - \alpha \D \grad f(x) [u])]. 
	\end{equation*}
	The differential of $\Retr$ at $(x, 0) \in \T\calM$ is $\D \Retr(x, 0)[(u, v)] = u + v$ for all $u, v \in \T_x \calM$
	\citep[Lem.~4.21]{boumal2020intromanifolds}.
	Therefore, for all $u \in \T_{x^\ast}\calM$ we have
	\begin{equation*}
		\D g(x^\ast)[u] = 
		u - \alpha \D \grad f(x^\ast)[u] = 
		u - \alpha \nabla_u \grad f = 
		(I - \alpha \Hess f(x^\ast))[u],
	\end{equation*}
	where in the second step we used the fact that the Riemannian connection $\nabla$ applied at a critical point of a vector field reduces to the usual differential of that vector field seen as a smooth map~\cite[Prop.~5.3]{boumal2020intromanifolds}.
\end{proof}

\begin{corollary} \label{cor:rgd-template-theorem}
	Let $f \colon \calM \to \reals$ be $C^2$.
	If the \eqref{eq:RGD} iteration map $g$ with $\alpha > 0$ has the Luzin $N^{-1}$ property, then $g$ avoids the strict saddles of $f$.
\end{corollary}

\begin{proof}
	Assume $x^\ast$ is a strict saddle point of $f$.
	Then, $\Hess f(x^\ast)$ has at least one negative eigenvalue. 
	Let $v$ be the unit eigenvector corresponding to such a negative eigenvalue $\lambda < 0$.
	By Lemma~\ref{lemma:Dgx-critical}, $\D g(x^\ast) = I - \alpha \Hess f(x^\ast)$, so $v$ is also an eigenvector of $\D g(x^\ast)$ with eigenvalue $1 - \alpha \lambda > 1$.
	The conclusion then follows from Theorem~\ref{thm:saddle-avoidance-template}.
\end{proof}

\subsection{The Riemannian case: real-analytic retractions} \label{section:analytic}

In Corollary~\ref{cor:rgd-template-theorem}, we confirmed that the \eqref{eq:RGD} iteration map avoids strict saddles provided it has the Luzin~$N^{-1}$ property.
To verify this condition, we adapt Theorem~\ref{thm:euclidean-gd-result} and the associated tools from the Euclidean to the Riemannian case.
We show that if the manifold and the retraction are real analytic, then the iteration map
\begin{equation*}
	g_\alpha(x) = \Retr_x(-\alpha \grad f(x))
\end{equation*}
satisfies the Luzin $N^{-1}$ condition for almost all $\alpha \in \reals$.

Real-analytic maps are particularly useful in this context because, like polynomials, they do not vanish on large sets of points unless they are identically zero.
We exploit this together with the fact that the map $\alpha \mapsto g_\alpha(x)$ is real analytic (even if $f$ itself is only $C^2$) in order to conclude that, for a fixed $x$, the set of values of $\alpha$ for which $\D g_\alpha(x)$ is singular has measure zero.

Let us review a few basic notions before stating the main theorem of this section.

\begin{definition}[{\protect{\citet[\S2.1, \S2.2]{krantz2002primer}}}] \label{def:realanaliticmapRn}
	Let $U \subseteq \Rm, V \subseteq \Rn$ be open sets. 
	A map $F \colon U \rightarrow V$ is \emph{real analytic} if every 
	point $x \in U$ has a neighborhood $W \subseteq U$ such that $F$ can be 
	expressed as a convergent power series in $W$.
	More precisely, letting $F(x) = (F_1(x), \dots, F_n(x))$, for each coordinate function $F_i$ 
	of $F$, there exists a power series centered at $x$ that converges to $F_i(z)$ for all 
	$z \in W$:
	\begin{equation*}
		F_i(z) = \sum_{\alpha} c_{\alpha} (z - x)^\alpha.
	\end{equation*}
	Here, $c_\alpha$ is a real coefficient, $\alpha = (\alpha_1, \dots, \alpha_m)$ is a multi-index of nonnegative integers and $(z-x)^\alpha$ represents the monomial~$(z_1 - x_1)^{\alpha_1} \cdots (z_m - x_m)^{\alpha_m}$. 
\end{definition}

Similarly to the smooth case, a \emph{real-analytic manifold} is a topological manifold 
with a maximal real-analytic atlas~\citep[Ch.~1]{hirsch1976differential}.
See also~\citep[\S2.7 and \S6.4]{krantz2002primer} for other equivalent definitions.
If a manifold $\calM$ is real analytic, then so is its tangent bundle $\T\calM$.

A map $F \colon \calM \rightarrow \calN$ between real-analytic manifolds is \emph{real analytic} if, for every point $x \in \calM$, there exist real-analytic charts $(U, \varphi)$ around $x$ and $(V, \psi)$ around $F(x)$ with $F(U) \subseteq V$ such that the coordinate representation $\psi \circ F \circ \varphi^{-1} \colon \varphi(U) \to \psi(V)$ is a real-analytic map in the sense of Definition~\ref{def:realanaliticmapRn}.

\begin{theorem} \label{thm:analytic-saddle-avoidance}
	Let $\calM$ be a Riemannian manifold with a retraction $\Retr \colon \T\calM \to \calM$.
	Assume $\calM$ and $\Retr$ are real-analytic.
	Let $f \colon \calM \to \reals$ be $C^2$.
	Then, for almost all $\alpha \in \reals$, the \eqref{eq:RGD} iteration map for $f$ satisfies the Luzin $N^{-1}$ property;
	in particular, \eqref{eq:RGD} avoids the strict saddles of $f$ for almost all $\alpha > 0$.
\end{theorem}

A few remarks are in order before presenting a proof of this theorem.

\begin{remark} \label{remark:rgd-analytic-requires-retraction-condition}
	It is currently unclear whether the real-analyticity assumption in Theorem~\ref{thm:analytic-saddle-avoidance} is necessary.
	Nevertheless, \emph{some} condition on the retraction is needed: without it, we can construct a counter-example.
	Indeed, if $f$ has a strict saddle point $x^*$, then one can construct a retraction that maps a large region of the tangent	bundle to $x^*$.
	For instance, suppose $\calM$ is a Hadamard manifold (i.e., complete, simply connected, with nonpositive sectional curvature).
	Then, the exponential retraction is globally defined and has an inverse.
	Choose a bounded open set $U \subseteq \calM$ such that $\inf_{x \in U} \| \grad f(x) \| \geq G > 0$, where $\| \cdot \|$ denotes the norm induced by the Riemannian metric.
	Select $0 < d_1 < d_2 < G$.
	Let $q \colon [0, \infty) \to [0, 1]$ be a smooth transition function such that $q(t) = 1$ for $t \leq d_1^2$ and $q(t) = 0$ for $t \geq d_2^2$.
	Then, define the retraction as
	\begin{equation*}
		\Retr_x(v) = 
		\Exp_{x^\ast}\!\Bigg( q \big(\| v \|^2 \big) \; \Log_{x^\ast}\!\big( \Exp_x(v) \big) \Bigg).
	\end{equation*}
	With this retraction and any $\alpha \geq 1$, we have $g(x) = \Retr_x(-\alpha \grad f(x)) = x^\ast$ for all $x \in U$, 
	so clearly $g$ does not avoid strict saddles.
\end{remark}

\begin{remark} \label{rem:analytic-retraction}
	To check that the retraction is real analytic, it is often convenient to check that it admits a real-analytic extension.
	For example, on the unit sphere $\Sn = \{x \in \Rn : \|x\| = 1\}$, the retraction $\Retr_x(v) = \frac{x+v}{\|x+v\|} = \frac{x+v}{\sqrt{1+\|v\|^2}}$ is real analytic because the map $(x, v) \mapsto \frac{x+v}{\sqrt{1+\|v\|^2}}$ is real analytic on $\Rn \times \Rn$ and the tangent bundle of $\Sn$ is a real-analytic submanifold of $\Rn \times \Rn$.
	\arxivonly{Some details follow (arxiv version).
	
	Let $\overline{\calN}$ and $\calM$ be real-analytic manifolds.
	Assume $\calN$ is a real-analytic submanifold of $\overline{\calN}$, that is, $\calN \subseteq \overline{\calN}$ is a real-analytic manifold such that the inclusion map $i \colon \calN \to \overline{\calN} \colon x \mapsto i(x) = x$ is real analytic.
	If $\bar{F} \colon \overline{\calN} \to \calM$ is real analytic, then the restriction $F = \bar{F}|_\calN$ is also real analytic.
	Indeed, for all $x \in \calN$ there exist real-analytic charts $(U, \varphi)$ of $\calN$ around $x$ and $(V, \psi)$ of $\calM$ around $F(x)$ such that $F(U) \subseteq V$.
	Thus, $\psi \circ F \circ \varphi^{-1} = \psi \circ \bar{F} \circ i \circ \varphi^{-1} \colon \varphi(U) \to \psi(V)$,	which is real analytic by composition.

	To study retractions as in the example above, we apply that general statement to $F = \Retr$ with $\calN = \T\calM$ together with the fact that if $\calM$ is a real-analytic submanifold of a linear space $\calE$ then $\T\calM$ is a real-analytic submanifold of $\calE \times \calE$.}
\end{remark}

\begin{remark}
	Theorem~\ref{thm:analytic-saddle-avoidance} is a particular instance of a more general result, as follows.
	Let $\calM$ be a smooth manifold and consider a family of maps $g_\delta \colon \calM \to \calM$ with parameter $\delta \in \Rk$.
	Assume $(x, \delta) \mapsto g_\delta(x)$ is continuous in $\delta$ and $C^1$ in $x$ and that for almost all $x \in \calM$ the set 
	\begin{equation*}
		\{ \delta \in \Rk \mid \D g_\delta(x) \textrm{ is singular\,} \}
	\end{equation*}
	has measure zero in $\Rk$.
	Then, $g_\delta$ avoids its unstable fixed points for almost all $\delta$. 
	
	This follows from Theorem~\ref{thm:saddle-avoidance-template}, since the assumptions imply that almost all maps $g_\delta$ satisfy the Luzin $N^{-1}$ property.
	This can be shown with a suitable adaptation of Lemma~\ref{lemma:vertical-horizontal-slices}.
\end{remark}

The remainder of this section constitutes a proof of Theorem~\ref{thm:analytic-saddle-avoidance}.

The following result extends Lemma~\ref{lemma:vertical-horizontal-slices} to the product space 
$\calM \times \reals$.
While the statement holds more generally for arbitrary measurable subsets $S$ of a product of measure spaces,
we restrict our attention to the specific case needed here, as it is easier to state.
To avoid the technicalities of defining measures on manifolds, we instead carry out the proof using coordinate charts as this is sufficient to define measure zero sets.

\begin{lemma} \label{lemma:vertical-horizontal-slices-mfds}
	Let $\calM$ be a smooth manifold and let $S$ be a closed subset of $\calM \times \reals$. 
	The following statements are equivalent:
	\begin{enumerate}
		\item $S$ has measure zero. 
		\item For almost all $x \in \calM$, the slice $S^x = \{ \alpha \in \reals \mid (x, \alpha) \in S \}$
		has measure zero in $\reals$.
		\item For almost all $\alpha \in \reals$, the slice $S_\alpha = \{ x \in \calM \mid (x, \alpha) \in S \}$
		has measure zero in $\calM$.
	\end{enumerate}
\end{lemma}
\begin{proof}
	Let $n = \dim(\calM)$.
	Consider a chart $(U, \varphi)$ of $\calM$, and let $V = \varphi(U) \subset \Rn$.
	Then $(U \times \reals, \varphi \times \mathrm{id})$ is a chart for $\calM \times \reals$, with image $V \times \reals \subset \Rn \times \reals$.
	Define the image of (part of) $S$ through that chart as
	\begin{equation*}
		\tilde{S} = (\varphi \times \mathrm{id})(S \cap (U \times \reals)) \subset V \times \reals.
	\end{equation*}
	Since $S$ is closed in $\calM \times \reals$ by assumption, the set $S \cap (U \times \reals)$ is closed in $U \times \reals$.
	Because $\varphi \times \mathrm{id}$ is a homeomorphism, $\tilde{S}$ is closed in $V \times \reals$.
	Therefore, $\tilde{S}$ is a measurable subset of $\Rn \times \reals$.
	
	For $z \in V$ and $x = \varphi^{-1}(z)$, the vertical slice of $\tilde{S}$ at $z$ is
	\begin{equation*}
		\tilde{S}^z = 
		\{ \alpha \in \reals \mid (z, \alpha) \in \tilde{S} \} = 
		\{ \alpha \in \reals \mid (x, \alpha) \in S \} =
		S^x.
	\end{equation*}
	For $\alpha \in \reals$, the horizontal slice is
	\begin{equation*}
		\tilde{S}_\alpha = 
		\{ z \in V \mid (z, \alpha) \in \tilde{S} \} = 
		\{ \varphi(x) \mid x \in U \text{ and } (x, \alpha) \in S \} = 
		\varphi(S_\alpha \cap U).
	\end{equation*}
	By Lemma~\ref{lemma:vertical-horizontal-slices}, the following are equivalent:
	\begin{enumerate}
		\item $\tilde{S}$ has measure zero in $\reals^{n+1}$.
		\item For almost all $x \in U$, the slice $S^x$ has measure zero in $\reals$.
		\item For almost all $\alpha \in \reals$, the slice $S_\alpha \cap U$ has measure zero in $\calM$.
	\end{enumerate}
		
	To conclude, observe that $\calM$ admits a countable atlas $\{ (U_i, \varphi_i) \}_{i \geq 0}$, and the union of the 
	corresponding chart domains $U_i \times \reals$ covers $\calM \times \reals$.
\end{proof}

In the proof of Theorem~\ref{thm:analytic-saddle-avoidance}, we use the previous lemma for the set $S$ of tuples $(x, \alpha)$ such that the differential of the iteration map is singular.
That set is indeed closed.
\begin{lemma} \label{lemma:singular-set-closed}
	Let $\calM$ be a Riemannian manifold with $C^1$ retraction $\Retr \colon \T\calM \to \calM$. 
	Let $f \colon \calM \to \reals$ be $C^2$ and denote the \eqref{eq:RGD} iteration map with step size $\alpha$ by $g_\alpha$. 
	Then, the set 
	\begin{equation*}
		S = \{ (x, \alpha) \in \calM \times \reals \mid \D g_\alpha(x) \textrm{ is singular\,}\}
	\end{equation*}
	is closed in $\calM \times \reals$.
\end{lemma}
\begin{proof}
	The following set is closed as it is defined by continuous equality conditions:
	\begin{equation*}
		S' = 
		\{ (x, v, \alpha) \in \T \calM \times \reals \mid \D g_\alpha(x)[v] = 0 \text{ and } \| v \| = 1 \}.
	\end{equation*}
	Let $\{(x_t, \alpha_t)\}_{t \geq 0}$ be a sequence in $S$ and assume that 
	$\lim_{t \to \infty} (x_t, \alpha_t) = (x, \alpha)$ is its limit in $\calM \times \reals$.
	To show that $S$ is closed, it suffices to show that $(x, \alpha)$ is in $S$.
	Since $(x_t, \alpha_t) \in S$, for each $t$, there exists $v_t \in \T_{x_t} \calM$ such that $(x_t, v_t, \alpha_t)$ is in $S'$.
	Let $U$ be a chart domain around $x$ so that the tangent bundle of $\calM$ restricted to $U$ can be identified with $U \times \Rn$.
	In particular, for $(x, v, \alpha) \in S'$ with $x \in U$ we identify $v$ with an element of $\Sn = \{u \in \Rn : \|u\| = 1\}$.
	As $x_t$ converges to $x$, there exists an index $k$ such that $x_t \in U$ for all $t \geq k$.
	Since $\Sn$ is compact, we can pass to a subsequence indexed by $t_i$ such that $(x_{t_i}, v_{t_i})$ converges: let $(x, v)$ denote the limit.
	Then, the sequence $\{ (x_{t_i}, v_{t_i}, \alpha_{t_i}) \}_{i \geq 0}$ with elements in $S'$ has a limit $(x, v, \alpha)$, and this limit is in $S'$, as the set $S'$ is closed.
	Therefore, $(x, \alpha) \in S$.
\end{proof}

In the next lemma, we argue for a fixed $x$ that the set of parameters $\alpha$ such that $\D g_\alpha(x)$ is singular has measure zero.
Previously, we observed in the Euclidean case that this set coincides with the roots of a non-zero polynomial.
While this structure does not generally carry over to the manifold setting, we can still show that the singular set is locally the zero level set of a nonzero real-analytic function---hence that it has measure zero.
Morally, the proof relies on the fact that the determinant of the differential of $g_\alpha$ is real analytic in $\alpha$, though some care is needed because the determinant is only well defined after specifying bases for the tangent spaces at $x$ and $g_\alpha(x)$; those bases need to depend on $\alpha$ in a real-analytic fashion, hence the argument needs to be executed locally then patched together.

\begin{lemma} \label{lemma:vertical-slices-Dg-singular-mfd}
	Let $\calM$ be a Riemannian manifold with a retraction $\Retr \colon \T\calM \to \calM$.
	Assume $\calM$ and $\Retr$ are real-analytic.
	Let $f \colon \calM \to \reals$ be $C^2$ and denote the \eqref{eq:RGD} iteration map with step size $\alpha$ by $g_\alpha$. 
	Then, for all $x \in \calM$, the set 
	\begin{equation*}
		S^x \coloneqq 
		\{ \alpha \in \reals \mid \D g_\alpha (x) \textrm{ is singular\,} \}
	\end{equation*}
	has measure zero in $\reals$.
\end{lemma}
\begin{proof}
	Since $f$ is $C^2$ and the map $\Retr$ is real analytic, the map
	\begin{equation*}
		(x, \alpha) \mapsto g_\alpha(x) = \Retr_x\!\left( -\alpha \grad f(x) \right)
	\end{equation*}
	is $C^1$ in $x$ and real analytic in $\alpha$.
	Fix an arbitrary point $x \in \calM$ and consider a real-analytic chart $(U_x, \varphi_x)$ around $x$.
	For some $\alpha \in \reals$, let $y = g_\alpha(x)$ and choose a real-analytic chart $(U_\alpha, \varphi_\alpha)$ around $y$.
	The set $\{ a \in \reals \mid g_a(x) \in U_\alpha \}$ is open and it contains $\alpha$, so we can extract an open interval
	$I_{\alpha}$ from it such that $\alpha \in I_{\alpha}$.
	
	Let us define a function $F_\alpha \colon I_{\alpha} \to \reals$ such that $F_\alpha(a) = 0$ exactly when $\D g_a(x)$ is singular.
	To do so, we represent the linear operator $\D g_a(x) \colon \T_x \calM \to \T_{g_a(x)} \calM$ as a matrix and define $F_\alpha(a)$ as its determinant.

	Explicitly, let $u_1, \ldots, u_n$ form a basis of $\T_x \calM$, and let $e_1, \ldots, e_n$ be the standard basis of $\Rn$.
	Define the matrix-valued map $M_\alpha \colon I_\alpha \to \reals^{n \times n}$ with elements given by 
	\begin{equation*}
		[M_\alpha(a)]_{ji} = e_j^\top \D \varphi_\alpha \big( g_a(x) \big) \big[\D g_a(x)[u_i] \big].
	\end{equation*}
	Observe that $\D g_a(x)$ is singular if and only if $M_\alpha(a)$ is singular.
	Indeed, if there exists some nonzero $v \in \T_x \calM$ such that $\D g_a(x)[v] = 0$, then for all $j$ in $1, \ldots, n$ it holds that 
	\begin{equation*}
		e_j^\top \D \varphi_\alpha \big( g_a(x) \big) \big[\D g_a(x)[v] \big] = 0 .
	\end{equation*}
	Expanding $v = \sum_{i=1}^n c_i u_i$, the above is equivalent to $M_\alpha(a)[c] = 0$ for some nonzero vector $c \in \Rn$.
	For the reverse direction, assume there exists a non-zero vector $v = \sum_{i=1}^n c_i u_i$ in the kernel of $M_{\alpha}(a)$.
	This implies that $[M_{\alpha}(a) c]_j = 0$ for all $j$.
	At the same time, we have that 
	\begin{align*}
		[M_{\alpha}(a) c]_j &= 
		\sum_{i=1}^n [M_{\alpha}(a)]_{ji} c_i \\ &= 
		\sum_{i=1}^n c_i e_j^\top \D \varphi_\alpha \big( g_a(x) \big) \big[\D g_a(x)[u_i] \big] = 
		e_j^\top \D \varphi_\alpha \big( g_a(x) \big) \big[\D g_a(x)[v] \big],
	\end{align*}
	hence $\D \varphi_\alpha \big( g_a(x) \big) \big[\D g_a(x)[v] \big] = 0$, and the conclusion follows since $\D \varphi_{\alpha}(g_a(x))$ is invertible.

	Now set 
	\begin{equation*}
		F_{\alpha} \colon I_{\alpha} \to \reals, \qquad 
		F_{\alpha} (a) = \det (M_{\alpha}(a)).
	\end{equation*}
	By composition, $F_\alpha$ is real analytic.
	Define the set of zeros of $F_\alpha$ as 
	\begin{equation}
		Z_\alpha =
		\{ a \in I_\alpha \mid F_\alpha(a) = 0 \} = 
		\{ a \in I_\alpha \mid \D g_a(x) \text{ is singular} \},
	\end{equation}
	where we used that $\D g_a(x)$ is singular if and only if $M_\alpha(a)$ is singular.
	From the Identity Theorem for real-analytic functions \cite[Cor.~1.2.7]{krantz2002primer}, we obtain that 
	\begin{equation}
		\textrm{either } \mu\big( Z_\alpha \big) = 0,
		\qquad\qquad \textrm{ or } Z_\alpha = I_\alpha.
		\label{eq:id-thm-1d-conclusion}
	\end{equation}
	Since the sets $\{ I_\alpha \}_{\alpha \in \reals}$ form an open cover of $\reals$, we use Lindel\"{o}f's
	lemma to extract a countable subcover $\{ I_{\alpha_i} \}_{i \geq 0}$.
	If $\mu(Z_{\alpha_i}) = 0$ for all $i \geq 0$, then immediately we obtain 
	\begin{equation}
		0 \leq 
		\mu \Big( \{ a \in \reals \mid \D g_a (x) \textrm{ is singular\,} \} \Big) \leq 
		\sum_{i \geq 0} \mu(Z_{\alpha_i}) = 
		0,
	\end{equation}
	as announced.

	Otherwise, there exists an index $j$ such that $Z_{\alpha_j} = I_{\alpha_j}$.
	Let us show that this cannot happen.
	The strategy is to deduce that $\D g_a(x)$ is singular for all $a \in \reals$, then to derive a contradiction from the fact that $g_0(x) = x$ for all $x \in \calM$, so that $\D g_0(x) = I$ is not singular. 
	
	Let $J$ be the non-empty set containing all indices $j \geq 0$ for which $Z_{\alpha_j} = I_{\alpha_j}$.
	We first show that, in fact, $J = \naturals$.
	For contradiction, assume that $J \neq \naturals$.
	Define the sets 
	\begin{equation*}
		Q_1 = \bigcup_{j \in J} I_{\alpha_j}
		\qquad\qquad \textrm{ and } \qquad\qquad
		Q_2 = \bigcup_{i \in \naturals \setminus J} I_{\alpha_i}.
	\end{equation*}
	Note that both sets are open and $Q_1 \cup Q_2 = \reals$.

	If $Q_1 \cap Q_2$ is non-empty, we can find some $j \in J$ and
	$i \in \naturals \setminus J$ such that $I_{\alpha_j} \cap I_{\alpha_i} \neq \emptyset$.
	This implies that $Z_{\alpha_i}$ contains the open set $I_{\alpha_j} \cap I_{\alpha_i} $.
	From~\eqref{eq:id-thm-1d-conclusion}, we deduce that $Z_{\alpha_i} = I_{\alpha_i}$ and hence $i$ is in $J$: a contradiction.

	Therefore, $Q_1 \cap Q_2$ is empty.
	Then, each set is the complement of the other one, so that they are both simultaneously open and closed.
	The only such subsets of $\reals$ are the empty set and $\reals$ itself.
	Since $Q_1$ is non-empty by assumption, it follows that $Q_1 = \reals$, that is, $\D g_a(x)$ is singular for all $a \in \reals$.
	However, this contradicts $\D g_0(x) = I$.
\end{proof}

The proof of Theorem~\ref{thm:analytic-saddle-avoidance} is now straightforward:
the set
\begin{equation*}
	S = \{ (x, \alpha) \in \calM \times \reals \mid \D g_\alpha(x) \textrm{ is singular\,}\}
\end{equation*}
is closed by Lemma~\ref{lemma:singular-set-closed}.
From Lemma~\ref{lemma:vertical-slices-Dg-singular-mfd}, for each $x \in \calM$ the vertical slice
\begin{equation*}
	S^x = \{ \alpha \in \reals \mid \D g_\alpha(x) \textrm{ is singular\,}\}
\end{equation*}
has measure zero in $\reals$.
Then, applying Lemma~\ref{lemma:vertical-horizontal-slices-mfds}, we obtain that, for almost all $\alpha \in \reals$,
\begin{equation*}
	\mu \Big( \{ x \in \calM \mid \D g_\alpha(x) \textrm{ is singular\,} \} \Big) = 0.
\end{equation*}
By Theorem~\ref{thm:luzin-property}, it follows that $g_\alpha$ has the Luzin $N^{-1}$ property for almost all $\alpha \in \reals$.
The conclusion then follows from Corollary~\ref{cor:rgd-template-theorem}.

\section{Saddle avoidance with a backtracking line-search method} \label{section:line-search}

The results above pertain to general dynamical systems with a single, regular iteration map $g$, together with applications to GD with constant step size.
In this section, we establish a saddle avoidance result for GD with a backtracking line-search method, as described in Algorithm~\ref{algo:bkt-stab}.
This is (Riemannian) GD with a \emph{modified} version of the standard backtracking line-search method using the Armijo criterion~\citep[Alg.~3.1]{nocedal2006optimization}.

We emphasize that our modification only deviates from the usual line-search procedure when the gradient norm falls below an arbitrary threshold $\varepsilon > 0$.
In this regime, the selected step sizes become non-increasing.
This effectively ensures that the step sizes stabilize (that is, are asymptotically constant) if the sequence of iterates converges.
That is a key ingredient in our proof of Theorem~\ref{thm:modified-armijo}.

Let us also highlight that the gradient of $f$ need not be Lipschitz continuous.
Even if it is (with some constant $L$), the algorithm is allowed to entertain (and possibly select) arbitrarily large step sizes (in particular, larger than $2/L$).

\begin{theorem} \label{thm:modified-armijo}
	Let $f \colon \calM \to \reals$ be $C^2$ and assume that either of the following holds:
	\begin{enumerate}
		\item $\calM = \Rn$ and $\Retr_x(v) = x + v$; or (more generally)
		\item $\calM$ is a real-analytic manifold with a Riemannian metric and a real-analytic retraction \\ $\Retr \colon \T\calM \to \calM$.
	\end{enumerate}
	Then, for all $\tau, r \in (0, 1), \varepsilon > 0$ and almost all $\bar{\alpha} > 0$, Algorithm~\ref{algo:bkt-stab} avoids the strict saddles of $f$.
\end{theorem}
\begin{proof}
	If $f$ has no strict saddles, the theorem is trivially true.
	Thus, assume $f$ has at least one strict saddle point.
	The first case ($\calM = \Rn$ with $\Retr_x(v) = x+v$) is a special case of the second one, hence we prove the latter.

	Iterate $x_{t+1}$ is obtained from $x_t$ by a gradient step with step size $\alpha_t$.
	Each step size $\alpha_t$ belongs to the discrete set $\{ \bar{\alpha}, \tau \bar{\alpha}, \tau^2 \bar{\alpha}, \ldots \}$.
	Accordingly, define the iteration maps $g_{(i)} \colon \calM \to \calM$ for $i = 0, 1, 2, \ldots$ corresponding to (R)GD with step sizes $\alpha_{(i)} = \tau^i \bar{\alpha}$, that is, $g_{(i)}(x) = \Retr_x(-\tau^i \bar{\alpha} \, \grad f(x))$.

	It follows from Theorem~\ref{thm:analytic-saddle-avoidance} that, for \emph{almost all} $\bar{\alpha}$, the iteration maps $g_{(i)}$ \emph{all} have the Luzin $N^{-1}$ property.
	Indeed, the theorem says $g_\alpha(x) = \Retr_x(-\alpha \grad f(x))$ has the stated property for all $\alpha$ except those in a measure zero set $A \subset \reals$; thus, $g_{(i)}$ has the property provided $\tau^i \bar{\alpha}$ is not in $A$; this holds for all $i$ as long as $\bar{\alpha}$ is not in $A \cup \tau^{-1} A \cup \tau^{-2} A \cup \cdots$, which is still a measure zero set.
	Going forward, consider this to be the case.
	Under the same conditions, Theorem~\ref{thm:analytic-saddle-avoidance} provides for each map $g_{(i)}$ a measure zero set $W_i$ with the following property: if $\lim_{t \to \infty} g_{(i)}^t(x)$ is a strict saddle point of $f$, then $x$ is in $W_i$.

	Assume the iterates $\{ x_t \}_{t \geq 0}$ generated by Algorithm~\ref{algo:bkt-stab} converge to a strict saddle point of $f$.
	We aim to show that $x_0$ belongs to a measure zero set.
	The proof proceeds in two main steps.
	First, we show that the step sizes selected by Algorithm~\ref{algo:bkt-stab} eventually stabilize.
	This allows us to apply the CSMT to each iteration map $g_{(i)}$, one of which corresponds to the stabilized step size: this provides local control. 
	Second, we use the Luzin $N^{-1}$ property of all iteration maps to enable global control.


	Let $x^\ast := \lim_{t \to \infty} x_t$ denote the strict saddle.
	Since $x_t$ converges to $x^\ast$, there exists some index $T \geq 0$ such that all subsequent iterates $x_t$ with $t \geq T$ remain in a closed ball $B_\rho(x^\ast)$ of radius $\rho > 0$ around $x^\ast$, where $\rho$ is chosen small enough so that:
	(a) the ball is compact (for example, $\rho$ strictly less than the injectivity radius of $\calM$ at $x^*$), and
	(b) $\| \grad f(x) \| < \varepsilon$ for all $x \in B_{\rho}(x^\ast)$. 
	
	We now show that for iterations $t \geq T$, the step sizes are lower-bounded, hence only a finite number is used.
	This follows since the gradient of $f$ is \emph{locally} Lipschitz continuous.
	Consider the compact set
	\begin{equation*}
		\calT = \Big\{ (x, s) \in \T \calM \mid x \in B_\rho(x^\ast) \text{ and } \| s \| \leq \bar{\alpha} \| \grad f(x) \| \Big\}.
	\end{equation*}
	There exists $L > 0$ such that for all $(x, s) \in \calT$ it holds that 
	\begin{equation} \label{eq:LipschitzlikeinproofABT}
		f(\Retr_x(s)) \leq f(x) + \langle s, \grad f(x) \rangle	+ \frac{L}{2} \| s \|^2.
	\end{equation}
	(See for example \citep[Lem.~10.57]{boumal2020intromanifolds}; in the Euclidean case, $L$ is a local Lipschitz constant for the gradient.)
	Therefore, the step sizes selected by Algorithm~\ref{algo:bkt-stab} satisfy the lower bound
	\begin{equation*}
		\alpha_t \geq 
		\min\!\Bigg( \alpha_T, \frac{2\tau (1-r)}{L} \Bigg) \eqqcolon \alpha_{\min},  \qquad\qquad \textrm{ for all } t \geq T.
	\end{equation*}
	(See for example \citep[Lem.~4.12]{boumal2020intromanifolds}; this is because any step size smaller than $\frac{2(1-r)}{L}$ satisfies the Armijo condition due to~\eqref{eq:LipschitzlikeinproofABT}.)
	This implies that only finitely many step sizes in the interval $[\alpha_{\min}, \bar{\alpha}]$ are used, and hence at least one step size is used infinitely many times.
	
	Additionally, for $t \geq T$, the step sizes are non-increasing.
	Indeed, since all iterates $x_t$ remain in $B_\rho(x^\ast)$ for all $t \geq T$, and $\| \grad f(x) \| < \varepsilon$ for all $x \in B_\rho(x^\ast)$, the mechanism in line~\ref{line:non-increasing-step} of Algorithm~\ref{algo:bkt-stab} ensures that the first step size considered when determining $\alpha_{t+1}$ is $\alpha_t$, rather than $\bar{\alpha}$.
	As a result, we have $\alpha_t \geq \alpha_{t+1}$ for all $t \geq T$.

	Therefore, since the set of allowable step sizes is finite and the sequence is eventually non-increasing, the sequence must eventually become constant.
	Thus, there exist an exponent $i \geq 0$ and some index $K \geq T$ such that $x_{t+1} = g_{(i)}(x_t)$ for all $t \geq K$.
	In particular,
	\begin{equation*}
		\lim_{t \to \infty} g_{(i)}^t(x_K) = \lim_{t\to\infty} x_t = x^*.
	\end{equation*}
	Since $x^\ast$ is a strict saddle of $f$, it follows from the discussion above that $x_K$ belongs to the measure zero set $W_i$.

	Of course, $i$ and $K$ may depend on $x_0$ (among other things).
	To remove the dependency on $i$, relax the statement to: $x_K$ belongs to $Z_0 := W_0 \cup W_1 \cup W_2 \cup \cdots$, which is also measure zero.
	Since $x_K$ was obtained from $x_{K-1}$ by applying one of the iteration maps $g_{(j)}$, and since all iteration maps have the Luzin $N^{-1}$ property, it follows that $x_{K-1}$ also belongs to a measure zero set, namely, $Z_1 := \bigcup_{j \geq 0} \, g_{(j)}^{-1}(Z_0)$ (where the union over $j$ accounts for the fact that we do not know $j$).
	Iterate this reasoning $K$ times to find a measure zero set $Z_K$ which contains $x_0$.
	Explicitly:
	\begin{equation*}
		Z_0 = \bigcup_{i \geq 0} W_i, \qquad\qquad \textrm{ and } \qquad\qquad
		Z_{k+1} = \bigcup_{j \geq 0} g_{(j)}^{-1}(Z_k), \textrm{ for } k \geq 0.
	\end{equation*}
	Since $K$ is unknown, relax the statement once again to: $x_0$ belongs to $Z := Z_0 \cup Z_1 \cup Z_2 \cup \cdots$.
	Each $Z_k$ is measure zero, hence $Z$ is measure zero: the proof is complete.
\end{proof}

\begin{remark} \label{remark:line-search-almost-any-tight}
	The statement ``for almost all $\bar{\alpha}$'' in Theorem~\ref{thm:modified-armijo} cannot be relaxed to \emph{all} $\bar{\alpha}$.
	To see this, consider again the function $f$ built in Remark~\ref{remark:luzin-not-removable}.
	It easily follows from the discussion in Remark~\ref{remark:euclidean-gd-counterexample} that running Algorithm~\ref{algo:bkt-stab} on this function with $\bar{\alpha} = 1$ and any $x_0$ such that $\|x_0\| > 2$ leads to $x_1 = 0$ (because the step size $\alpha_0 = \bar{\alpha}$ is accepted), that is, the algorithm converges to a strict saddle point of $f$ in one iteration, and it does so for a large set of initial points $x_0$.
\end{remark}

\begin{theorem} \label{thm:soc-convergence}
	Under the same assumptions as Theorem~\ref{thm:modified-armijo}, and assuming additionally that $ f $ is real analytic, or, more generally, that it satisfies the Kurdyka--{\L}ojasiewicz property, the following holds: for all $ \tau, r \in (0, 1)$, all $\varepsilon > 0$, all $ \bar{\alpha} > 0 $ and all $x_0 \in \calM $, the iterates produced by Algorithm~\ref{algo:bkt-stab} either leave every compact subset of $ \calM $, or converge to a point $ x^\ast \in \calM $.
	
	Moreover, for all $\tau, r \in (0, 1)$, all $\varepsilon > 0$, and for almost all $\bar{\alpha} > 0$ and almost every $x_0 \in \calM$, if the limit point $x^\ast$ exists, it satisfies $\grad f(x^\ast) = 0$ and $\Hess f(x^\ast) \succeq 0$.
\end{theorem}
\begin{proof}[Sketch of proof]
	We first adapt the argument from \cite{absil2005iteratesanalytic} to the Riemannian setting. 
	Since the proof is local in nature, the extension from the Euclidean case is straightforward.
	Using \cite[Thm.~4.1]{absil2005iteratesanalytic}, we can show that Algorithm~\ref{algo:bkt-stab} satisfies the strong descent conditions of \cite[Def.~3.1]{absil2005iteratesanalytic}.
	By \cite[Thm.~3.2]{absil2005iteratesanalytic}, it then follows that either the iterates leave every compact subset of $\calM$, or they converge to a single limit point $x^\ast \in \calM$.
	
	Assume the latter.
	Then, there exists some compact set $\calK \subseteq \calM$ such that the sequence $\{ x_k \}_{k \geq 0}$ remains in $\calK$ at all times. 
	By \cite[Lem.~10.57]{boumal2020intromanifolds}, the set $\calT = \{ (x, v) \in \T \calM \mid x \in \calK, \| v \| \leq \bar{\alpha} \| \grad f(x) \| \}$ is compact in $\T \calM$ and there exists a constant $L$ such that the inequality from Eq.~\eqref{eq:LipschitzlikeinproofABT} holds for all $(x, s) \in \calT$.
	This implies that the step sizes $\alpha_t$ are uniformly lower bounded \citep[Lem.~4.12]{boumal2020intromanifolds}. 
	It follows that $\lim_{k \to \infty} \| \grad f(x_k) \| = 0$ \cite[Cor.~4.13]{boumal2020intromanifolds}, which shows that $\grad f(x^\ast) = 0$.
	Finally, from Theorem~\ref{thm:modified-armijo}, we obtain that for almost all $\bar{\alpha} > 0$ and almost any initial $x_0$, the limit point $x^\ast$ satisfies $\Hess f(x^\ast) \succeq 0$.
\end{proof}

\section{RGD with constant step size: exponential retraction} \label{section:additional-results-c2}

We now return to Riemannian gradient descent with a \emph{constant step size}, motivated by two key considerations.
First, when $\calM$ is \emph{not real analytic}, the results above no longer apply.
Second, while the earlier analysis shows that \eqref{eq:RGD} avoids strict saddles for \emph{almost all} step sizes, it does not provide an \emph{explicit interval} of acceptable step sizes.
In contrast, in the Euclidean setting, if $\nabla f$ is $L$-Lipschitz continuous, then \emph{any} step size $\alpha \in (0, 1/L)$ yields saddle avoidance.

In this section, we provide guarantees for \eqref{eq:RGD} with all step sizes in an interval, in the case when the retraction is the \emph{exponential map}~\citep[Ch.~5]{lee2018riemannian}.
The exponential map at a point $x \in \calM$, denoted by $\Exp_x \colon \T_x \calM \to \calM$, is defined as $\Exp_x(v) = \gamma_v(1)$, where $\gamma_v$ is the unique geodesic starting at $x$ with initial velocity $v \in \T_x \calM$.
The RGD iteration map $g \colon \calM \to \calM$ is then
\begin{align} \tag{ExpRGD} \label{eq:ExpRGD}
	g(x) = \Exp_x(- \alpha \grad f(x)).
\end{align}
We assume throughout this section that $\calM$ is complete, that is, $\Exp_x$ is defined on the whole tangent space $\T_x \calM$ (see Remark~\ref{remark:exp-real-analytic}).

As shown in Corollary~\ref{cor:rgd-template-theorem}, strict saddle avoidance holds if $g$ satisfies the Luzin $N^{-1}$ property, which by Theorem~\ref{thm:luzin-property} is equivalent to $\D g(x)$ being rank deficient only on a measure zero set.
Since $g$ involves the Riemannian gradient of $f$, the differential $\D g(x)$ depends on the Riemannian Hessian $\Hess f(x)$.
To control it, as usual, we assume that $f$ is $C^2$ and that $\grad f$ is $L$-Lipschitz continuous.
The latter can be stated as a bound on the Hessian.

\begin{definition} \label{def:grad-L-lipschitz-alternative}
	Let $f \colon \calM \rightarrow \reals$ be $C^2$.
	The vector field $\grad f$ is \emph{$L$-Lipschitz continuous} if 
	\begin{equation*}
		\| \Hess f(x) [u] \| \leq L \| u \| \qquad\qquad \textrm{ for all } x \in \calM, u \in \T_x \calM.
	\end{equation*}
	(See \citep[Cor.~10.48]{boumal2020intromanifolds} for alternative characterizations.)
\end{definition}

Before moving forward with our saddle avoidance results, we take a necessary detour to study $\D g(x)$.
When $x$ is a critical point of the cost function, the situation is straightforward, as demonstrated in the following lemma.
\begin{lemma} \label{lem:Dg-not-singular-critical-point}
	Let $f \colon \calM \to \reals$ be $C^2$ and assume that $\grad f$ is $L$-Lipschitz continuous. 
	If $x^\ast$ is a critical point of $f$ then the differential $\D g(x^\ast)$ with $0 \leq \alpha < 1/L$ is invertible.
\end{lemma}
\begin{proof}
	If $x^\ast$ is a critical point of $f$, we know from Lemma~\ref{lemma:Dgx-critical} that $\D g(x^\ast) = I - \alpha \Hess f(x^\ast)$.
	Thus, for all nonzero $u \in \T_{x^\ast} \calM$ it holds that $\|\D g(x^\ast)[u]\| = \|u - \alpha \Hess f(x^\ast)[u]\| \geq \|u\| - \alpha L \|u\|> 0$, that is, $\D g(x^\ast)[u]$ is nonzero.
\end{proof}

When $x$ is not a critical point, the computations are markedly more involved.
In the next section, we introduce the necessary tools to analyze the differential of the iteration map at non-critical $x$.

\subsection{Jacobi fields and the differential of the iteration map}
\label{sec:JacobiDgx}

We need to study $\D g(x)$ with $g$ as in~\eqref{eq:ExpRGD}.
This requires us to differentiate through the Riemannian exponential.
\emph{Jacobi fields} are ideal for this purpose.

A Jacobi field is a vector field along a geodesic that satisfies a certain second-order differential equation involving the curvature of the manifold~\citep[Ch.~10]{lee2018riemannian}.
Intuitively, these fields describe infinitesimal variations of geodesics, making them a natural tool to access the differential of the exponential map.
We begin by reviewing standard concepts from Riemannian geometry, and gradually build toward an expression for $\D g(x)$ stated in Lemma~\ref{lemma:Dgx}.

\begin{definition} \cite[Ch.~6]{lee2018riemannian}
	Let $\calM$ be a Riemannian manifold.
	The \emph{injectivity radius at a point} $x \in \calM$, denoted by $\inj(x)$, is the supremum of all $r > 0$ such that $\Exp_x$ is defined and is a diffeomorphism from the ball $B_r(x) = \{ v \in \T_x \calM \mid \| v \| < r \}$ onto its image. 

	The \emph{injectivity radius of the manifold} $\calM$ is $\inj(\calM) = \inf_{x \in \calM} \inj(x)$.
\end{definition}

\begin{definition} \cite[Ch.~10]{lee2018riemannian}
	Let $\calM$ be a Riemannian manifold and let $(x, v) \in \T \calM$.
	Define the geodesic segment $\gamma \colon [0, 1] \rightarrow \calM$ by $\gamma(t) = \Exp_x(tv)$ and let $y = \gamma(1)$. 
	We say that \emph{$y$ is conjugate to $x$ along $\gamma$} if $v$ is a critical point of $\Exp_x$, that is, if $\D\Exp_x(v)$ is singular.
\end{definition}

The following lemma is an immediate consequence of these definitions.
\begin{lemma} \label{lemma:no-conjugate-points-inj}
	Let $\calM$ be a Riemannian manifold and let $(x, v) \in \T \calM$.
	Define $\gamma \colon [0, 1] \rightarrow \calM$, $\gamma(t) = \Exp_x(t v)$.
	Then, if $\| v \| < \inj(x)$, there are no conjugate points along $\gamma$.
\end{lemma}

Later, we will relate the differential $\D g(x)$ to the Hessian of a certain squared distance function.
For this Hessian to be well defined, we first need to ensure that the squared distance function is indeed smooth.
To this end, we use the following lemma.
\begin{lemma} \cite[Cor.~10.25]{boumal2020intromanifolds} \label{lemma:dist-smooth}
	Let $\calM$ be a Riemannian manifold.
	The map $(x, v) \mapsto (x, \Exp_x(v))$ is a diffeomorphism from $\{ (x, v) \in \T\calM \mid \|v\| < \inj(x) \}$ onto its image, namely, $\{ (x, y) \in \calM \times \calM \mid \dist(x, y) < \inj(x) \}$.
	Its smooth inverse is the map
	\begin{equation*}
		(x, y) \mapsto (x, \Log_x(y)),
	\end{equation*}
	where $\Log_x$ (the \emph{Riemannian logarithm}) serves as an inverse of $\Exp_x$ on the stated domains.

	In particular, the squared distance function to $y \in \calM$, defined by $x \mapsto \dist(x, y)^2 = \| \Log_x(y) \|^2$, is smooth on $\{ x \in \calM \mid \dist(x, y) < \inj(x) \}$.
\end{lemma}	

We now return to Jacobi fields.
These are solutions to a linear, second-order differential equation along a geodesic segment $\gamma \colon [0, 1] \to \calM$.
The space of solutions forms a $2n$-dimensional vector space, with $n = \dim \calM$.
Each solution $J$ is a vector field along $\gamma$ uniquely determined by its initial value $J(0) \in \T_{\gamma(0)}\calM$ and initial velocity $\D_t J(0) \in \T_{\gamma(0)}\calM$, where $\D_t$ denotes the covariant derivative with respect to $t$ along $\gamma$.
With this, we have a natural basis for decomposing any Jacobi field, as formalized next.

\begin{lemma}[{\protect{\citet[Prop.~14.30]{villani2009optimal}}}] \label{lemma:jacobi-fields-decomposition}
	Let $\calM$ be an $n$-dimensional Riemannian manifold and let $\gamma \colon [0, 1] \to \calM$ be a geodesic.
	Pick an arbitrary orthonormal basis $\{ e_1, \dots, e_n \}$ of $\T_{\gamma(0)} \calM$.
	Define the parallel-transported frame $E_i(t) = \mathrm{PT}_{t \leftarrow 0}^{\gamma} (e_i)$, forming an orthonormal
	basis of $\T_{\gamma(t)} \calM$ for all $t \in [0, 1]$.
	Let $J_0$ and $J_1$ be matrix-valued fields whose $n$ columns are Jacobi fields along $\gamma$,
	expressed in the basis $\{E_i(t)\}_{i=1}^n$, and satisfying the initial conditions:
	\begin{equation*}
		J_0(0) = 0, \qquad
		\D_t J_0(0) = I_n, \qquad
		J_1(0) = I_n, \qquad
		\D_t J_1(0) = 0.
	\end{equation*}
	Then, for any $t \in [0, 1]$, any Jacobi field $J(t)$ along $\gamma$ can be expressed as 
	\begin{equation*}
		J(t) = J_1(t) \; J(0) + J_0(t) \; \D_t J(0),
	\end{equation*}
	where the ``matrix product'' is understood through the coordinate frames.
\end{lemma}

The next standard result allows to check invertibility of $J_0(t)$.
The proof relies on the link between Jacobi fields and the differential of the exponential map.

\begin{lemma} \label{lemma:J_0_invertible} \citep[Lem.~10.20]{lee2018riemannian} 
	Let $\calM$ be a Riemannian manifold and let $\gamma \colon [0, 1] \to \calM$ be a geodesic.
	Define the matrix-valued field $J_0$ as in Lemma~\ref{lemma:jacobi-fields-decomposition}.
	Then, $J_0(t)$ is invertible for all $t \in (0, 1]$ if and only if there are no conjugate points along $\gamma$.
\end{lemma}
\begin{proof}
	Let $(x, v) \in \T\calM$ such that $\gamma(t) = \Exp_x(tv)$.
	By~\citep[Prop.~10.10]{lee2018riemannian}, for every $w \in \T_x \calM$, the Jacobi field along $\gamma$ such that $J(0) = 0$ and $\D_t J(0) = w$ is given by 
	\begin{equation} \label{eq:J_0}
		J(t) = \D\Exp_x(tv)[tw].
	\end{equation}
	Consider now the basis $\{e_1, \dots, e_n\}$ of $\T_x \calM$ from Lemma~\ref{lemma:jacobi-fields-decomposition}.
	The $i$th column of $J_0$ is the Jacobi field from~\eqref{eq:J_0} corresponding to $w = e_i$, so that we have $J_0(t) = t \D\Exp_x(tv)$ (in frame coordinates).
	From this expression, it is clear that $J_0(t)$ is invertible for all $t \in (0, 1]$ exactly when there are no conjugate points along $\gamma$.
\end{proof}

In the following lemma, we show a connection between the Hessian of the half-squared 
distance function and Jacobi matrices along a geodesic.
The result can be found without proof in \citep{villani2011regularity} and \citep{villani2009optimal}, while a sketch appears in~\citep[\S2.1]{pennec2018suppA}.
For completeness, we include a proof here.

\begin{lemma} \label{lemma:hessian-half-squared-distance-function}
	Let $\calM$ be a complete Riemannian manifold.
    Let $x \in \calM$ and $v \in \T_x \calM$.
	Consider the geodesic $\gamma \colon [0, 1] \rightarrow \calM$, $\gamma(t) = \Exp_x(t v)$ and let $y = \gamma(1)$.
	Define the matrix-valued fields $J_0, J_1$ as in Lemma~\ref{lemma:jacobi-fields-decomposition}.
	Then, if $\|v\| < \inj(x)$, we have that 
    \begin{equation*}
        J_0(1)^{-1} J_1(1) =
		\Hess \Big( x \mapsto \frac{1}{2} \dist(x, y)^2 \Big)(x) .
    \end{equation*} 
\end{lemma}
\begin{proof}
Let $\dot{x} \in \T_x \calM$ with $\| \dot{x} \| = 1$.
Since $\inj(x) > \|v\|$ and the injectivity radius is continuous, let $\bar\varepsilon > 0$ be such that
$\inj(\Exp_x(u\dot{x})) > \|v\|$ for all $u \in [-\bar\varepsilon,\bar\varepsilon]$.
Define the geodesic $\sigma \colon \reals \to \calM$ given by $\sigma(s) = \Exp_x(s\dot{x})$, and let
\begin{equation}
	\delta \coloneqq  \min_{u \in [-\bar\varepsilon, \bar\varepsilon]} \inj(\sigma(u)) - \|v\|  > 0,
	\qquad \qquad
	\varepsilon =
	\min\left\{ \bar\varepsilon, \, \frac{1}{2} \delta \right\} > 0.
\end{equation}
Consider the one-parameter family of curves $ \Gamma \colon (-\varepsilon, \varepsilon) \times [0,1] \to \calM $ given by
\begin{equation}
	\Gamma(s, t) = \Exp_{\sigma(s)} \bigl( t \Log_{\sigma(s)}(y) \bigr).
\end{equation}
We will show that $\dds \Gamma(s,t)\big\vert_{s=0}$ is a Jacobi field.
First, observe that $\Gamma$ is smooth: the exponential map is globally defined and smooth, while for the term involving the logarithm, for any $\left| s \right| < \varepsilon$, it holds that
\begin{equation}
	\dist(\sigma(s), y) \leq
	\dist(\sigma(s), x) + \dist(x, y) \leq
	|s| \, \| \dot{x} \| + \| v \| < 
	\varepsilon + \| v \| \leq
	\frac{\delta}{2} + \| v \| <
	\inj(\sigma(s)).
\end{equation}
From Lemma~\ref{lemma:dist-smooth}, we conclude that $s \mapsto \Log_{\sigma(s)}(y)$ is smooth, which confirms that $\Gamma$ is smooth.

Now, observe that $\Gamma$ is a variation of $\gamma$ through geodesics. 
Indeed, it is a \emph{variation of $\gamma$} since $\Gamma$ is smooth on its domain and $\Gamma(0, t) = \gamma(t)$.
Moreover, it is a \emph{variation through geodesics} since all the main curves $\Gamma_s(t) \coloneq \Gamma(s, t)$ obtained by fixing $s$ are geodesics.
It follows that $J(t) \coloneq \dds \Gamma(s, t) \big\vert_{s=0}$ is a Jacobi field along $\gamma$.
Following the proof of Prop.~10.4 in \citep{lee2018riemannian}, the initial conditions of $J$ are given by
\begin{equation*}
	J(0) = \sigma'(0) = \dot{x}, \textrm{ and }
\end{equation*}
\begin{equation*}
	\D_t J(0) = \left. \D_s \Log_{\sigma(s)}(y) \right|_{s = 0} = 
	- \Hess \Big( x \mapsto \frac{1}{2} \dist(x, y)^2 \Big)(x)[\dot{x}].
\end{equation*}
Using Lemma~\ref{lemma:jacobi-fields-decomposition}, we can write this Jacobi field $J$ as
\begin{equation} \label{eq:jacobi-field-squared-distance}
	J(t) = 
	J_1(t) \; \dot{x} - J_0(t) \; \Hess \Big( x \mapsto \frac{1}{2} \dist(x, y)^2 \Big)(x)[\dot{x}].
\end{equation}
We can also compute $J(1)$ directly:
\begin{equation*}
	J(1) = 
	\dds \Gamma(s, 1) \bigg\vert_{s=0} = 
	\dds \Exp_{\sigma(s)}(\Log_{\sigma(s)}(y)) \bigg\vert_{s=0} = 
	\dds y \bigg\vert_{s=0} =
	0.
\end{equation*}
Combining this with~\eqref{eq:jacobi-field-squared-distance} evaluated at $t=1$, we obtain that
\begin{equation} \label{eq:J_1_hess_sq_dist}
	0 = 
	J(1) = 
	J_1(1) \; \dot{x} - J_0(1) \; \Hess \Big( x \mapsto \frac{1}{2} \dist(x, y)^2 \Big)(x)[\dot{x}].
\end{equation}
Since we assumed $\| v \| < \inj(x)$, from Lemma~\ref{lemma:no-conjugate-points-inj} we know that there are no conjugate points along $\gamma$.
It follows from Lemma~\ref{lemma:J_0_invertible} that $J_0(1)$ is invertible.
Since $\dot{x}$ was arbitrary, the conclusion follows from~\eqref{eq:J_1_hess_sq_dist}.
\end{proof}

To obtain an expression for $\D g(x)$, we write $\D g(x)[\dot x]$ as $J(1)$ for some Jacobi field $J$ to be constructed, then we decompose it via Lemma~\ref{lemma:jacobi-fields-decomposition}.
This makes $J_0, J_1$ appear: we eliminate them partially via Lemma~\ref{lemma:hessian-half-squared-distance-function}, which is how a squared distance function appears in the expression.
This is advantageous because it is easily controlled later on in terms of the curvature of the manifold.

\begin{lemma} \label{lemma:Dgx}
	Let $f \colon \calM \to \reals$ be $C^2$ on a complete Riemannian manifold $\calM$.
	Fix $x \in \calM$ and let $\alpha > 0$ satisfy $\alpha \| \grad f(x) \| < \inj(x)$.
	Define the curve $\gamma \colon [0, 1] \to \calM$ by 
	\begin{equation*}
		\gamma(t) = \Exp_x(-\alpha t \grad f(x)).
	\end{equation*}
	Notice its relation to~\eqref{eq:ExpRGD}, i.e., the iteration map $g(x) = \Exp_x(-\alpha \grad f(x))$.

	Fix $y = \gamma(1) = g(x)$ and define $h \colon \calM \to \reals$ as the half-squared distance to $y$, that is,
	\begin{equation*}
		h(z) = \frac{1}{2} \dist(z, y)^2.
	\end{equation*}
	Then, the differential of the iteration map $g$ satisfies
	\begin{equation}
        \D g(x) = 
        J_0(1) \Big( \Hess\,h(x) - \alpha \Hess f(x) \Big).
		\label{eq:DgxJzeroinverseetc}
  	\end{equation}
	(Recall from Lemma~\ref{lemma:Dgx-critical}  that if $x$ is critical for $f$ then $\D g(x) = I - \alpha \Hess f(x)$.)
\end{lemma}

\begin{proof}
Let $\dot{x} \in \T_x \calM$.
Let $I$ be an open interval around $0$ and define a smooth curve $\sigma \colon I \to \calM$ 
with $\sigma(0) = x, \sigma'(0) = \dot{x}$.
Since $g$ is $C^1$, we can write its differential as 
\begin{equation}
	\D g(x)[\dot{x}] = 
	\dds g(\sigma(s)) \bigg\vert_{s=0} = 
	\dds \Exp_{\sigma(s)} \Big(-\alpha \grad f(\sigma(s)) \Big) \bigg\vert_{s=0}.
\label{eq:Dgx-first}
\end{equation}
To relate this to a Jacobi field, we consider the variation of $\gamma$ through geodesics
\begin{equation*}
	\Gamma \colon I \times [0, 1] \to \calM, \qquad\qquad
	\Gamma(s, t) = \Exp_{\sigma(s)}( t V(s)), 
\end{equation*}
where we defined $V$ as the vector field on $\sigma$ given by $V(s) = -\alpha \grad f(\sigma(s))$. 
Note that indeed $\Gamma$ is $C^1$ everywhere, and is a variation of $\gamma$ since 
\begin{equation*}
	\Gamma(0, t) = \Exp_{\sigma(0)}(t V(0)) = \Exp_x (-t \alpha \grad f(x)) = \gamma(t).
\end{equation*}
It is also a variation through geodesics since the curves $\Gamma_s(t)$ obtained by fixing $s$ are geodesics.
The variation field $\dds \Gamma(s, t) \big\vert_{s=0} \coloneq J(t)$ is a Jacobi field along $\gamma$, and from the proof of \cite[Prop.~10.4]{lee2018riemannian}, we identify its initial conditions as 
\begin{equation*}
	J(0) = \dot{x}, \textrm{ and} 
\end{equation*}
\begin{equation*}
	\D_t J(0) = 
	\D_s V(0) = 
	-\alpha \D_s \grad f(\sigma(s)) \big\vert_{s=0} = 
	-\alpha \Hess f(x)[\dot{x}] . 
\end{equation*}
Notice that $\D g(x)[\dot{x}]$ is this Jacobi field evaluated at $t=1$.
Indeed, from \eqref{eq:Dgx-first}, we have 
\begin{equation*}
	\D g(x)[\dot{x}] =
	\dds \Exp_{\sigma(s)}(-\alpha \grad f(\sigma(s))) \Big\vert_{s=0} = 
	\dds \Gamma(s, 1) \Big\vert_{s=0} = 
	J(1).
\end{equation*}
It remains to apply Lemma~\ref{lemma:jacobi-fields-decomposition} to expand $J(1)$ as follows:
\begin{equation*}
	\D g(x)[\dot{x}] = J(1) = 
	J_1(1) \dot{x} - \alpha J_0(1) \Hess f(x)[\dot{x}] = 
	J_0(1) \Big( \Hess\,h(x) - \alpha \Hess f(x) \Big)[\dot x].
\end{equation*}
In the second step above, we used the identity $J_0(1)^{-1} J_1(1) = \Hess\,h(x)$, which follows from 
Lemma~\ref{lemma:hessian-half-squared-distance-function} owing to the assumption $\alpha \| \grad f(x) \| < \inj(x)$.
\end{proof}

We now have the necessary ingredients to study invertibility of $\D g(x)$. 
The first step is to ensure $\alpha \| \grad f(x) \| < \inj(x)$ so that the squared distance function $h$ is smooth.
With this, Lemmas~\ref{lemma:J_0_invertible} and~\ref{lemma:Dgx} together show that $\D g(x)$ is singular if and only if $\Hess \big( h - \alpha f \big)(x)$ is singular.
Our approach in the remainder of this section is to impose a strict upper bound on $\alpha$ to ensure that this Hessian remains non-singular \emph{everywhere}.
(In light of Theorem~\ref{thm:luzin-property}, this may be more conservative than necessary.)

\subsection{The case of Hadamard manifolds} \label{section:hadamard}

A \emph{(Cartan--)Hadamard manifold} is a Riemannian manifold that is complete and simply connected and whose sectional curvature is non-positive~\citep[Ch.~12]{lee2018riemannian}.
When the cost function $f$ on such a manifold is $C^2$ with an $L$-Lipschitz gradient, we show that \eqref{eq:ExpRGD} avoids strict saddles for all step sizes $\alpha \in (0, 1/L)$.
Since $\Rn$ is a Hadamard manifold, this is a strict generalization from the Euclidean case.

\begin{theorem} \label{thm:hadamard-strict-saddle-avoidance}
  Let $f \colon \calM \to \reals$ be $C^2$ on a Hadamard manifold $\calM$.
	Assume $\grad f$ is $L$-Lipschitz. 
	Then, if $0 < \alpha < 1/L$, \eqref{eq:ExpRGD} avoids the strict saddles of $f$.
\end{theorem}
\begin{proof}
	Let $g(x) = \Exp_x(-\alpha\grad f(x))$ as in~\eqref{eq:ExpRGD}.
	Below, we show that $\D g(x)$ is never singular.
	Theorem~\ref{thm:luzin-property} then implies that $g$ has the Luzin $N^{-1}$ property, and the conclusion follows from Corollary~\ref{cor:rgd-template-theorem}.

	Fix some $x \in \calM$ and define $\gamma \colon [0, 1] \to \calM$, $\gamma(t) = \Exp_x(- \alpha t \grad f(x))$.
	Let $y = \gamma(1) = g(x)$.
	The half-squared distance function $h \colon \calM \to \reals$ defined by $h(z) = \frac{1}{2} \dist(z, y)^2$ is smooth on all of $\calM$~\citep[Prop.~12.9]{lee2018riemannian} (this can also be seen from Lemma~\ref{lemma:dist-smooth} and the fact that the injectivity radius of a Hadamard manifold is infinite).

	If $x$ is a critical point of $f$, then $\D g(x)$ is invertible by Lemma~\ref{lem:Dg-not-singular-critical-point}.
	If $x$ is not critical, then we use Lemma~\ref{lemma:Dgx} as follows.
	Geodesics on Hadamard manifolds do not have conjugate points~\citep[Thm.~11.12]{lee2018riemannian}, so that Lemma~\ref{lemma:J_0_invertible} tells us $J_0(1)$ in~\eqref{eq:DgxJzeroinverseetc} is invertible.
	Therefore, $\D g(x)$ is singular if and only if
	\begin{equation*}
		\Hess\,h(x) - \alpha \Hess f(x)
	\end{equation*}
	is singular.
	The Hessian of the half-squared distance function satisfies $\Hess\,h(x) \succeq I$ since $\calM$ has non-positive sectional curvature~\citep[Thm.~6.6.1]{jost2017riemannian}.
	Also, $\alpha \Hess f(x) \prec I$ since the eigenvalues of $\Hess f(x)$ are at most $L$.
	Thus, the operator above is invertible (it is positive definite).
	We conclude that $\D g(x)$ is invertible for all $x$, as desired. 
\end{proof}

\arxivonly{Further exploiting the favorable properties of Hadamard manifolds, we can also derive a strict saddle avoidance result for the proximal point algorithm: see Appendix~\ref{app:proximalhadamard} (arxiv version).}

\subsection{The case of positive and mixed curvature}

The proof technique used above for Hadamard manifolds relies on the fact that their injectivity radius is infinite.
This generally fails for manifolds that have positive sectional curvature.
The next theorem makes up for this by further restricting the step size $\alpha$, to an interval that is smaller than $(0, 1/L)$.
Before getting to its proof, we provide two corollaries.
The second one regards the Stiefel manifold, which includes the sphere and the orthogonal group as particular cases.

\begin{theorem} \label{thm:positive-curvature}
	Let $(\calM, g)$ be a complete Riemannian manifold with sectional curvature upper bounded by $K_{\max} \in (0, \infty)$ and such that $\inj(\calM) \geq J > 0$.
	Assume $f \colon \calM \to \reals$ is $C^2$ with $L$-Lipschitz continuous gradient bounded as follows: $\sup_{x \in \calM} \| \grad f(x) \| \leq G < \infty$.
	Suppose that the step size satisfies:
	\begin{equation*}
		0 < \alpha < \min\!\Bigg\{ 
			\frac{JL}{G}, 
			\frac{L}{G \sqrt{K_{\max}}} \arccot\!\Bigg( \frac{L}{G \sqrt{K_{\max}} }\Bigg) 
		\Bigg\} \frac{1}{L}.
	\end{equation*}
	Then, \eqref{eq:ExpRGD} avoids the strict saddles of $f$.
\end{theorem}

Observe that for $K_{\max} \to 0$ and $J \to \infty$ we recover the step size restriction $\alpha < 1/L$ from Theorem~\ref{thm:hadamard-strict-saddle-avoidance} for Hadamard manifolds.
We presented the Hadamard case separately since the geometric properties there allowed for a more direct proof.

When sectional curvature is lower-bounded away from zero ($K_{\min} > 0$), the manifold's diameter is bounded.
Thanks to the Lipschitz property of $\grad f$, we can derive a uniform upper bound for $\| \grad f(x) \|$.
If the manifold satisfies additional conditions, a lower bound on the injectivity radius also holds.
In this case, we can make the condition on $\alpha$ more precise, as demonstrated in the following corollary.

\begin{corollary} \label{cor:pos-curvatures}
	Let $(\calM, g)$ be a complete, connected Riemannian manifold with sectional curvatures
	between $K_{\min} > 0$ and $K_{\max} > 0$.
	Assume $f \colon \calM \to \reals$ is $C^2$ and has $L$-Lipschitz continuous gradient and that one of 
	the following holds:
	\begin{itemize}
		\item $\calM$ is even dimensional and orientable, or
		\item $\calM$ is simply connected, $\dim(\calM) \geq 3$ and $K_{\min} > \frac{1}{4} K_{\max}$.
	\end{itemize}
	Then, \eqref{eq:ExpRGD} avoids the strict saddles of $f$ with all step sizes satisfying
	\begin{equation*}
		0 < \alpha < 
		\frac{1}{\pi} \sqrt{\frac{K_{\min}}{K_{\max}}} 
		\arccot\!\Bigg( \frac{1}{\pi} \sqrt{\frac{K_{\min}}{K_{\max}}} \Bigg) \frac{1}{L}.
	\end{equation*}
\end{corollary}
\begin{proof}
	First, if one of the two assumptions holds, then $\inj(\calM) \geq \frac{\pi}{\sqrt{K_{\max}}} =: J$---this is due to Klingenberg, see~\citep[Thm.~5.9 and~5.10]{cheeger1975comparison}.
	From the bound on the sectional curvature, we deduce a bound on the Ricci curvature,
	$R_c(v, v) \geq (n-1) K_{\min} > 0$. 
	Applying Myers Theorem~\cite[Thm.~12.24]{lee2018riemannian}, it follows that $\calM$ is compact with diameter $\diam(\calM) := \sup_{x, y \in \calM} \dist(x, y) \leq \frac{\pi}{\sqrt{K_{\min}}}$.
	In particular, $f$ has at least one critical point $\bar x$.
	It follows from Lipschitz continuity of the gradient that~\citep[Cor.~10.48, Thm.~10.9]{boumal2020intromanifolds}
	\begin{equation}
		\| \grad f(x) \| \leq
		L \; \dist(x, \bar x) \leq
		L \; \diam(\calM) \leq 
		\frac{L \pi}{\sqrt{K_{\min}}} =:
		G.
	\label{eq:grad-norm-bound-positive-curvature}
	\end{equation}
	Plug these values of $J$ and $G$ in Theorem~\ref{thm:positive-curvature} to conclude.
\end{proof}

For the Stiefel manifold, we obtain the following corollary of Theorem~\ref{thm:positive-curvature}

\begin{corollary} \label{cor:stiefel}
	Consider the Stiefel manifold $\St(n, p)$ with the Euclidean metric $\langle U, V \rangle = \trace(U^\top V)$.
	Let $f \colon \St(n, p) \to \reals$ be $C^2$ with $L$-Lipschitz continuous gradient. 
	Assume 
	\begin{equation*}
		0 < \alpha < 
			\frac{1}{\pi \sqrt{p}} \arccot\!\Bigg( \frac{1}{\pi \sqrt{p}} \Bigg) \frac{1}{L}.
 	\end{equation*}
	Then, \eqref{eq:ExpRGD} avoids the strict saddles of $f$.
	In particular, for a sphere ($p = 1$) this holds with $0 < \alpha \leq 0.4/L$.
\end{corollary}
\begin{proof}
	The result follows from Theorem~\ref{thm:positive-curvature} with appropriate values for $\Kmax$, $J$ and $G$.
	For curvature, we have $K_{\max} = 1$~\citep[Thm.~10]{zimmermann2025high}.
	For the injectivity radius, we use $\inj(\St(n, p)) = \pi =: J$~\citep[Thm.~3.6]{zimmermann2025injectivity}.
	For the gradient, we use the Lipschitz condition on the gradient (same as in~\eqref{eq:grad-norm-bound-positive-curvature}) to obtain that the gradient norm is bounded by $L \; \diam(\St(n, p)) \leq L \; \pi \sqrt{p} =: G$ owing to the diameter bound $\diam(\St(n, p)) \leq \pi \sqrt{p}$~\citep[Thm.~6.6]{mataigne2024bounds}.
	Theorem~\ref{thm:positive-curvature} then provides a range of step sizes which simplifies using $\arccot(t) \leq \pi/2$ for all $t \geq 0$.
\end{proof}

We now turn to the proof of Theorem~\ref{thm:positive-curvature}.

\begin{proof}[Proof of Theorem~\ref{thm:positive-curvature}]
	We restrict the proof to the case where the manifold is connected, ensuring that all distance functions are well-defined.
	However, the analysis extends to the general case of disconnected manifolds, since the algorithm remains confined to the connected component containing the initial point.
	This follows from the fact that the exponential map moves along continuous curves, and thus cannot transition between disconnected components.

	From Corollary~\ref{cor:rgd-template-theorem}, we know \eqref{eq:RGD} avoids strict saddles if the iteration map has the Luzin $N^{-1}$ property.
	Owing to Theorem~\ref{thm:luzin-property}, this is so in particular if $\D g(x)$ is invertible for all $x \in \calM$.
	We proceed to restrict the step size to ensure as much.

	Fix $x \in \calM$.
	Denoting the iteration map with step size $\alpha$ by $g_\alpha$, let $y_\alpha = g_\alpha(x)$.
	Define $h_\alpha \colon \calM \to \reals$, $h_{\alpha}(z) = \frac{1}{2} \dist(z, y_\alpha)^2$.

	Start by imposing $\alpha < \alpha_{\max} \coloneq J / G$.
	With this restriction, we ensure that $\alpha \| \grad f(x) \| < J \leq \inj(x)$.
	From Lemma~\ref{lemma:Dgx}, we can thus write
	\begin{equation} \label{eq:Dg-x-pf}
		\D g_{\alpha}(x) = 
		J_0(1) \Big( \Hess h_{\alpha}(x) - \alpha \Hess f(x) \Big),
	\end{equation}
	where $J_0(t)$ is a matrix whose columns are Jacobi fields, as defined in Lemma~\ref{lemma:jacobi-fields-decomposition}.
	Since $\alpha \| \grad f(x) \| < \inj(x)$, the matrix $J_0(1)$ is invertible by Lemma~\ref{lemma:J_0_invertible},
	so from Eq.~\eqref{eq:Dg-x-pf} we have 
	\begin{equation}
		\D g_{\alpha}(x) \textrm{ is singular } \qquad\qquad \iff \qquad\qquad
		\Hess h_{\alpha}(x) - \alpha \Hess f(x) \textrm{ is singular}.
		\label{eq:Dg-singular-equiv-cond}
	\end{equation}
	Since $K_{\max} > 0$, if $\alpha \| \grad f(x) \| < \pi / (2 \sqrt{K_{\max}})$, then for any $\dot{x} \in \T_x \calM$ we have
	the following bound on $\Hess h_{\alpha}(x)$~\citep[Thm.~6.6.1]{jost2017riemannian}:
	\begin{equation} \label{eq:hess_h_alpha_bound}
		\langle \dot{x}, \Hess h_{\alpha}(x)[\dot{x}] \rangle  \geq \alpha \sqrt{K_{\max}} \| \grad f(x) \| \, \cot\!\Big(\alpha \sqrt{K_{\max}} \| \grad f(x) \| \Big) \| \, \dot{x} \|^2.
	\end{equation}
	We cannot yet claim that $\Hess h_{\alpha}(x) - \alpha \Hess f(x)$ is invertible, like we did in the Hadamard case: this is because $t \cot(t)$ is smaller than 1 for $t \in (0, \pi)$.
	To proceed, we further restrict $\alpha$.

	Recall $\alpha_{\max} = J / G$ and define $\omega \colon \calM \to \reals$ by
	\begin{equation*} 
		\omega(x) =
		\alpha_{\max} \sqrt{K_{\max}} \| \grad f(x) \|.
	\end{equation*}
	By design, this function is uniformly bounded as
	\begin{equation} \label{eq:omega-bound}
		\omega(x) \leq J \sqrt{K_{\max}}.
	\end{equation}
	Consider a step size of the form $\alpha = \delta \alpha_{\max}$ for some $\delta \in (0, 1)$.
	Our goal is to determine the largest $\delta$ such that $\Hess\,h(x) - \alpha \Hess f(x)$ is invertible.
	To this end, we rely on the bound for the Hessian of the squared distance function given in~\eqref{eq:hess_h_alpha_bound}, which holds if we require 
	\begin{equation*}
		\delta \alpha_{\max} \| \grad f(x) \| < \frac{\pi}{2 \sqrt{K_{\max}}} .
	\end{equation*}
	This is satisfied if we pick $\delta$ in the interval 
	\begin{equation} \label{eq:first-delta-cond}
		\delta \in \Big( 0, \min\!\big( 1, \frac{\pi}{2 J \sqrt{K_{\max}}} \big) \Big).
	\end{equation}
	From this and~\eqref{eq:omega-bound}, it follows that $\delta \omega(x) \leq \delta J \sqrt{K_{\max}} < \pi / 2$.
	Combining this with the bound from Eq.~\eqref{eq:hess_h_alpha_bound}, and using that $t \to t \cot(t)$ is decreasing on $(0, \pi)$, for any $\dot{x}$ with $\| \dot{x} \| = 1$, we obtain that 
	\begin{equation} \label{eq:hess-h-bound-delta}
		\langle \dot{x}, \Hess h_{\alpha}(x) [\dot{x}] \rangle 
		  \geq \delta \omega(x) \cot(\delta \omega(x)) 
		  \geq \delta J \sqrt{K_{\max}} \cot(\delta J \sqrt{K_{\max}}).
	\end{equation}
	Moreover, since $\grad f$ is $L$-Lipschitz continuous and by the definition of $\alpha_{\max}$, we have 
	\begin{equation} \label{eq:hess-f-bound-delta}
		\alpha \langle \dot{x}, \Hess f(x)[\dot{x}] \rangle 
		\leq \delta \alpha_{\max} L
		= \delta \frac{J L}{G}.
	\end{equation}
	From~\eqref{eq:hess-h-bound-delta} and~\eqref{eq:hess-f-bound-delta}, it is clear that if 
	\begin{equation} \label{eq:delta-restriction}
		\frac{L}{G} < \sqrt{K_{\max}} \cot(\delta J \sqrt{K_{\max}}),
	\end{equation}
	then $\Hess h_{\alpha}(x) - \alpha \Hess f(x) \succ 0$.
	Since $x$ was arbitrary, recalling~\eqref{eq:Dg-singular-equiv-cond}, it follows that $\D g_{\alpha}(x)$ is invertible for all $x$ (as desired) provided we meet the assumptions made along the way.
	Specifically, to meet the condition for $\delta$ from~\eqref{eq:delta-restriction}, we impose 
	\begin{equation*}
		\delta < 
		\frac{1}{J \sqrt{K_{\max}}} \arccot\!\Big( \frac{L}{G \sqrt{K_{\max}} }\Big).
	\end{equation*}
	Observe that this condition readily ensures $ \delta < \pi / (2 J \sqrt{K_{\max}}) $ as required by~\eqref{eq:first-delta-cond}, since $\arccot(t) \leq \pi / 2$ for $t \geq 0$.
	Since we also require $\delta \in (0, 1)$, we conclude that \eqref{eq:ExpRGD} avoids strict saddles if the step size satisfies
	\begin{equation*}
		0 < \alpha = 
		\delta \alpha_{\max} < 
		\min\!\Bigg\{ \frac{JL}{G}, \frac{L}{G \sqrt{K_{\max}}} \arccot\!\Big( \frac{L}{G \sqrt{K_{\max}} }\Big) \Bigg\} \frac{1}{L},
	\end{equation*}
	as announced.
\end{proof}

\section{RGD with constant step size: projection retraction on spheres} \label{section:metric-projection-retraction}

The previous section provides saddle avoidance guarantees for \eqref{eq:RGD} with all constant step sizes in an interval, specifically for the exponential retraction.
If the manifold $\calM$ is embedded in a Euclidean space $\calE$, another popular choice is the \emph{metric projection retraction} defined by $\Retr_x(v) = \calP_\calM(x + v)$ where $\calP_\calM \colon \calE \to \calM$ is defined as
\begin{equation*}
	\calP_\calM(y) = \argmin{z \in \calM} \| z - y \|.
\end{equation*}
Here, $\|\cdot\|$ denotes the Euclidean norm on $\calE$.
In general, $\calP_\calM(y)$ is a set, but it reduces to a singleton on a relevant domain so that $\Retr$ as defined earlier is indeed a retraction (possibly restricted to small $v$)~\citep[\S5.12]{boumal2020intromanifolds}.

Accordingly, we consider saddle avoidance for the iteration map
\begin{align} \label{eq:ProjRGD} \tag{ProjRGD}
	g(x) = \calP_\calM(x - \alpha \grad f(x)),
\end{align}
whose domain is the interior of the set 
\begin{equation*}
	\{ x \in \calM \mid x - \alpha \grad f(x) \textrm{ is in the interior of the set where } \calP_\calM \text{ is single-valued}\}.
\end{equation*}
For compact $\calM$, this domain is all of $\calM$ provided $\alpha > 0$ is sufficiently small.
Some saddle avoidance results exist for compact $\calM$ (e.g., in \citep{lee2019strictsaddles}), though they typically require rather small step sizes.

In this section, we focus on the relevant special case where $\calM$ is a sphere $\mathbb{S}^{d} = \{ x \in \reals^{d+1} \mid \|x\| = 1 \}$ (or a product of spheres) with the submanifold metric $\inner{u}{v} = u^\top v$.
The metric projection retraction is defined globally by $\Retr_x(v) = \frac{x+v}{\|x+v\|}$.
We show that~\eqref{eq:ProjRGD} avoids strict saddles under the standard step size condition $0 < \alpha < 1/L$.

We first provide a general statement about the invertibility of $\D g(x)$, then we particularize to products of spheres.

\begin{lemma} \label{lem:Dg-expression-projRGD}
	Let $\calM$ be an embedded submanifold of a Euclidean space $\calE$ and let $f \colon \calM \to \reals$ be $C^2$.
	Then, for $x$ in the domain of $g$ and $\dot{x} \in \T_x \calM$ we have the equivalence 
    \begin{equation*}
        \D g(x)[\dot{x}] = 0 
		\qquad\qquad \iff \qquad\qquad
        \dot{x} - \alpha \D \grad f(x)[\dot{x}] \in \N_{g(x)} \calM,
    \end{equation*}
	where $\N_{g(x)} \calM$ denotes the \emph{normal space} to $\calM$ at $g(x)$. 
\end{lemma}
\begin{proof}
	Since $\calM$ is smooth, \citet{dudek1994nonlinearprojection} show that $\calP_\calM$ is smooth on the interior of the subset of $\calE$ where $\calP_\calM$ maps to a singleton.
	They further provide that, for any $y$ in that domain, the following holds:
	\begin{equation*}
		\ker \D \calP_\calM (y) = \Big( \T_{\calP_\calM(y)} \calM \Big)^\perp = \N_{\calP_\calM(y)} \calM.
	\end{equation*}
	By assumption, $y := x - \alpha \grad f(x) \in \calE$ is in said domain.
	Differentiating $g$ at $x$ along $\dot{x} \in \T_x \calM$ yields
	\begin{equation*}
			\D g(x)[\dot{x}] = 
			\D \calP_{\calM} (x - \alpha \grad f(x) ) \Big[ \dot{x} - \alpha \D \grad f(x)[\dot{x}] \Big].
	\end{equation*}
	Observe that $\calP_\calM(y) = g(x)$ to conclude.
\end{proof}

\begin{proposition} \label{prop:product-spheres-metric-projection}
	Let $\calM = \mathbb{S}^{d_1} \times \cdots \times \mathbb{S}^{d_N}$ be a product of spheres (not necessarily all of the same dimension).
	If $f \colon \calM \to \reals$ is $C^2$ and $\grad f$ is $L$-Lipschitz, then \eqref{eq:ProjRGD} with $0 < \alpha < \frac{1}{L}$ avoids the strict saddles of $f$.
\end{proposition}
\begin{proof}
For ease of notation, we give a proof for a product of two spheres.
The general case follows similarly.

By inspection, the domain of $g$ is the entire manifold.
For $(x, y) \in \mathbb{S}^{d_1} \times \mathbb{S}^{d_2}$, the iteration map factorizes as $g(x, y) = (g_1(x, y), g_2(x, y))$ with 
\begin{equation*}
    g(x, y) =
	(g_1(x, y), g_2(x, y)) =
	\Big( \mathcal{P}_{\mathbb{S}^{d_1}} (x - \alpha \grad_x f(x, y)), \mathcal{P}_{\mathbb{S}^{d_2}} (y - \alpha \grad_y f(x, y)) \Big),
\end{equation*}
where $\grad_x f(x, y)$ denotes the gradient of $x \mapsto f(x, y)$ for some fixed $y$ and analogously for $\grad_y f(x, y)$.

For contradiction, assume $\D g(x, y)$ is singular.
Then, Lemma~\ref{lem:Dg-expression-projRGD} provides $\dot{x} \in \T_x \mathbb{S}^{d_1}$ and $\dot{y} \in \T_y \mathbb{S}^{d_2}$ (not both zero) such that 
\begin{equation}
  (\dot{x}, \dot{y}) - \alpha \D \grad f(x, y)[(\dot{x}, \dot{y})]  
	\in \N_{g_1(x, y)} \mathbb{S}^{d_1} \times \N_{g_2(x, y)} \mathbb{S}^{d_2}.
	\label{eq:oblique-Dg-cond}
\end{equation}
Since $\N_{g_1(x, y)} \mathbb{S}^{d_1} = \text{span}\{ x - \alpha \grad_x f(x, y) \}$ and similarly for $\N_{g_2(x, y)} \mathbb{S}^{d_2}$, condition~\eqref{eq:oblique-Dg-cond} implies that there must exist $\beta, \gamma \in \reals $ such that:
\begin{equation}
    (\dot{x}, \dot{y}) - \alpha \D \grad f(x, y)[(\dot{x}, \dot{y})] = 
    \Big( \beta (x - \alpha \grad_x f(x, y)),
          \gamma (y - \alpha \grad_y f(x, y)) \Big).
\label{eq:oblique-Dg-cond-2}
\end{equation}
Apply $\Proj_{(x, y)}$ (the orthogonal projector to the tangent space of $\calM$ at $(x, y)$) on both sides of~\eqref{eq:oblique-Dg-cond-2}.
Since $\calM$ is a Riemannian submanifold, the projection of the differential of the gradient is the Riemannian Hessian~\citep[Cor.~5.16]{boumal2020intromanifolds}, so that we get
\begin{equation}
    (\dot{x}, \dot{y}) - \alpha \Hess f(x, y)[(\dot{x}, \dot{y})] = 
    \Big( - \alpha \beta \grad_x f(x, y), - \alpha \gamma \grad_y f(x, y) \Big).
\label{eq:proj-tangent-space}
\end{equation}
Similarly, if we apply $\Proj_{(x, y)}^\perp$ on both sides of \eqref{eq:oblique-Dg-cond-2}, we obtain that
\begin{equation}
    - \alpha \Proj_{(x, y)}^\perp \Big( \D \grad f(x, y)[(\dot{x}, \dot{y})] \Big) = 
	\Big( \beta x, \gamma y \Big).
\label{eq:proj-normal-space}
\end{equation}
From here, we rewrite the left-hand side with the Gauss formula \citep[Thm.~8.2]{lee2018riemannian} and the fact that, on product manifolds, the \emph{second fundamental form $\mathrm{I\!I}$} factorizes into the second fundamental forms of the individual manifolds.
More specifically, in our setup, this means that 
$\mathrm{I\!I}((\dot{x}_1, \dot{y}_1), (\dot{x}_2, \dot{y}_2)) = 
(\mathrm{I\!I}(\dot{x}_1, \dot{x}_2), \mathrm{I\!I}(\dot{y}_1, \dot{y}_2))$ for all tangent vectors $\dot{x}_1, \dot{x}_2 \in \T_x \mathbb{S}^{d_1}$ and all 
$\dot{y}_1, \dot{y}_2 \in \T_y \mathbb{S}^{d_2}$.
Then, (see also~\citep[\S5.11]{boumal2020intromanifolds} for computations with $\mathrm{I\!I}$)
\begin{align*}
    \Proj_{(x, y)}^\perp (\D \grad f(x, y)[(\dot{x}, \dot{y})]) 
			& = 
    \mathrm{I\!I}\Big( (\dot{x}, \dot{y}), \grad f(x, y) \Big) \\
			& = 
    - \Big( x \dot{x}^\top \grad_x f(x, y), y \dot{y}^\top \grad_y f(x, y) \Big).
\end{align*}
Plugging that in~\eqref{eq:proj-normal-space}, we obtain the relationship
\begin{equation*}
    \alpha \Big( x \dot{x}^\top \grad_x f(x, y), y \dot{y}^\top \grad_y f(x, y) \Big) = 
     \Big( \beta x, \gamma y \Big).
\end{equation*}
From this, we can identify $\beta = \alpha \dot{x}^\top \grad_x f(x, y)$ and 
$\gamma = \alpha \dot{y}^\top \grad_y f(x, y) $. 
Replacing these back in~\eqref{eq:proj-tangent-space} and taking an inner product with $(\dot x, \dot y)$, we obtain
\begin{multline*}
    \Big\langle (\dot{x}, \dot{y}), \big( I - \alpha \Hess f(x, y) \big)[(\dot{x}, \dot{y})] \Big\rangle = \\
    - \alpha^2 \Big( \dot{x}^\top \grad_x f(x, y) \grad_x f(x, y)^\top \dot{x} + 
    \dot{y}^\top \grad_y f(x, y) \grad_y f(x, y)^\top \dot{y} \Big) \leq 0 .
\end{multline*}
However, note that since $0 < \alpha < 1/L$, it must be that $I - \alpha \Hess f(x, y) \succ 0$.
The previous relation contradicts this, so $\D g(x, y)$ cannot be singular.
This holds for all $(x, y)$, hence $g$ has the Luzin $N^{-1}$ property by Theorem~\ref{thm:luzin-property}. 
The saddle avoidance conclusion then follows from Corollary~\ref{cor:rgd-template-theorem}.
\end{proof}

\section{Perspectives}

We end by listing several future directions:
\begin{itemize}
	\item \textbf{Beyond GD and simple algorithms.} Theorem~\ref{thm:modified-armijo} offers a first saddle avoidance result for a simple yet adaptive algorithm that includes discontinuous branching.
	It remains unclear how to do this for algorithms involving more complex conditional logic, mainly because of how discontinuities affect our ability to call on the CSMT.
	In particular, it remains open whether the \emph{standard} backtracking line-search algorithm (that is, Algorithm~\ref{algo:bkt-stab} with $\varepsilon=0$) avoids saddle points.

	\item \textbf{On the analyticity assumptions.} In Theorem~\ref{thm:modified-armijo}, we allow $f$ to be merely $C^2$ but we assume that $\Retr$ is real analytic.
	While some condition on the retraction is indeed necessary---as demonstrated by Remark~\ref{remark:rgd-analytic-requires-retraction-condition}---perhaps a weaker condition would do.

	\item \textbf{Better constant step size guarantees on manifolds.} We secured saddle avoidance for~\eqref{eq:ExpRGD} and~\eqref{eq:ProjRGD} for sufficiently small step sizes.
	We did so in a way that $\D g$ is nowhere singular.
	However, it would be enough to ensure $\D g$ is invertible almost everywhere.
	Would this extra leeway allow for larger step sizes?
\end{itemize}

\section*{Acknowledgments}

We thank Quentin Rebjock, Christopher Criscitiello and Ir\`ene Waldspurger for helpful discussions.
Chris helped in particular with the proof of Lemma~\ref{lemma:hessian-half-squared-distance-function}.

\section*{Funding and Conflicts of interests}
This work was supported by the Swiss State Secretariat for Education, Research and Innovation (SERI) under contract number MB22.00027.

\noindent The authors declare no conflicts of interest.

\bibliographystyle{abbrvnat}
\bibliography{../bibtex/boumal}

\appendix

\renewcommand{\thedefinition}{\thesection.\arabic{definition}}
\renewcommand{\thelemma}{\thesection.\arabic{lemma}}

\numberwithin{equation}{section}

\setcounter{definition}{0}
\setcounter{lemma}{0}

\section{About the Luzin $N^{-1}$ property} \label{appendix:proofs-luzin}

This section proves Theorem~\ref{thm:luzin-property}.
First, we recall facts about measurable sets and functions.
\begin{enumerate}
	\setlength{\itemsep}{0.3pt}
    \setlength{\parskip}{0pt}
	\item The preimage of any Borel (hence measurable) set under a continuous function is also measurable.
	\item Compact subsets of measurable sets are measurable.
	\item Countable unions and intersections of measurable sets are measurable.
	\item If a set $S$ has measure zero, then any subset of $S$ also has measure zero.
\end{enumerate}

\begin{definition}
	A point $x \in \calM$ is a \emph{critical point} of a differentiable map 
	$g \colon \calM \to \calM$ if the rank of $\D g(x)$ is not maximal.
	We use $\Sigma_g$ to denote the set of all critical points of $g$.
\end{definition}

With this definition, we can restate Theorem~\ref{thm:luzin-property} as follows:
If $g$ is continuously differentiable, then $g$ has the Luzin $N^{-1}$ property if and only if $\mu(\Sigma_g) = 0$.
We prove this after an auxiliary lemma that we think of as a ``non-collapse'' property (consider its contrapositive).

\begin{lemma}[\citet{ponomarev1987submersions}] \label{lemma:luzin-helper-lemma}
	For a continuous map $g \colon \calM \to \calM$, the following properties are equivalent: 
	\begin{itemize}
		\item $g$ has the Luzin $N^{-1}$ property.
		\item for all compact sets $K \subseteq \calM$, we have that 
		\begin{equation*}
			\mu(g(K)) = 0
			\qquad \implies \qquad 
			\mu(K) = 0.
		\end{equation*}
	\end{itemize}
\end{lemma}
\begin{proof}[Proof of Lemma~\ref{lemma:luzin-helper-lemma}, \citep{ponomarev1987submersions}]
	We prove both statements using contraposition. 

	For the first direction, assume that there exists a compact set $K$ such 
	that $\mu(K) > 0$ yet $\mu(g(K)) = 0$. 
	Letting $E = g(K)$, we see that $\mu(E) = 0$ yet $\mu(g^{-1}(E)) \geq \mu(K) > 0$, because 
	$g^{-1}(g(K)) \supseteq K$, so $g$ does not have the Luzin $N^{-1}$ property. 
	
	For the reverse direction, assume that $g$ does not have the Luzin $N^{-1}$ property.
	That is, there exists a set $E \subseteq \calM$ such that $\mu(E) = 0$ and either $g^{-1}(E)$ is not 
	measurable, or $\mu(g^{-1}(E)) > 0$.
	Since $\mu(E)=0$, there exists a Borel set $B \supseteq E$ such that $\mu(B)=0$ \cite[Lem.~1.5.3]{cohn2013measure}. Since $g$ is continuous, $g^{-1}(B)$ is measurable. Furthermore, $g^{-1}(B)$ is not of measure zero: otherwise, since $g^{-1}(E) \subseteq g^{-1}(B)$, the set $g^{-1}(E)$ would also have measure zero, a contradiction.
	The Lebesgue measure is inner regular \citep[Prop.~1.4.1]{cohn2013measure}, so we can always find a compact set $K \subseteq g^{-1}(B)$ such that $\mu(K) > 0$. 
	Then, $g(K) \subseteq B$ and since $\mu(B) = 0$, we also have $\mu(g(K)) = 0$. 
	Therefore, we constructed a compact set $K$ with $\mu(K) > 0$ and 
	$\mu(g(K)) = 0$. 
\end{proof}
	
\begin{proof}[Proof of Theorem~\ref{thm:luzin-property}, \citep{ponomarev1987submersions}]
	Let $\Sigma_g = \{ x \in \calM \mid \rank \D g(x) < \dim\calM \}$ be the 
	set of critical points of $g$ and assume $\mu(\Sigma_g) = 0$.
	We first show that $g$ has the Luzin $N^{-1}$ property using Lemma~\ref{lemma:luzin-helper-lemma}.
	Specifically, we show that if a compact set $K \subseteq \calM$ satisfies 
	$\mu(g(K)) = 0$, then $\mu(K) = 0$.

	Let $K$ be an arbitrary compact subset of $\calM$ such that $\mu(g(K)) = 0$.
	Denoting the complement of $\Sigma_g$ in $\calM$ by $\Sigma_g^c$, we can decompose $K$ into 
	\begin{equation*}
		K = (K \cap \Sigma_g) \cup (K \cap \Sigma_g^c) .
	\end{equation*}
	For each point $x$ in $\Sigma_g^c$, the inverse function theorem grants a neighborhood $U_x$ of $x$ such that $g \vert_{U_x}$ is a diffeomorphism onto its image $g(U_x)$. 
	The collection $\{ U_x \}_{x \in K \cap \Sigma_g^c }$ forms a (possibly uncountable) open cover of $K \cap \Sigma_g^c$.
	Using Lindel{\"o}f's lemma, we can select $x_0, x_1, x_2, \ldots$ in $K \cap \Sigma_g^c$ such that $\{ U_{x_i} \}_{i \geq 0}$ is a countable subcover of $K \cap \Sigma_g^c$.
	Then, we can write the set $K$ as  
	\begin{equation}
		K = 
		\Big( \Sigma_g \cap K \Big) \cup 
		\Big( \bigcup_{i \geq 0} ( U_{x_i} \cap K ) \Big).
		\label{eq:K-decomposition}
	\end{equation}
	Observe that since the set $\Sigma_g$ of critical points of $g$ is closed, the set $\Sigma_g \cap K$ is closed, too, so it is measurable.
	Similarly, $  \bigcup_{i \geq 0} (U_{x_i} \cap K) $ is measurable as a countable union of measurable sets.
	Using the decomposition of $K$ from \eqref{eq:K-decomposition}, we deduce that 
	\begin{equation}
		0 \leq 
		\mu(K) \leq 
		\mu\Big( \Sigma_g \cap K \Big) + \mu\Big( \bigcup_{i \geq 0} ( U_{x_i} \cap K ) \Big) = 
		\mu\Big( \bigcup_{i \geq 0} ( U_{x_i} \cap K ) \Big) \leq 
		\sum_{i \geq 0} \mu( U_{x_i} \cap K ),
		\label{eq:K-measure-bound}
	\end{equation}
	where in the second to last step we used that $\mu(\Sigma_g) = 0$, so $\mu(\Sigma_g \cap K) = 0$, too.
	Now, since $g \vert_{U_{x_i}}$ is a diffeomorphism onto its image, then in particular we have that 
	$g^{-1} \vert_{g(U_{x_i})}$ preserves sets of measure zero~\cite[Prop.~6.5]{lee2012smoothmanifolds}.
	Because we assumed $\mu(g(K)) = 0$, we also have $\mu(g(U_{x_i} \cap K)) = 0$ for all $i \geq 0$
	and thus also $\mu(U_{x_i} \cap K) = 0$ for all $i \geq 0$.
	Plugging this back in \eqref{eq:K-measure-bound}, we obtain that $\mu(K) = 0$.
	Since $K$ was an arbitrary compact subset of $\calM$ with $\mu(g(K)) = 0$, from 
	Lemma~\ref{lemma:luzin-helper-lemma} we obtain that $g$ has the Luzin $N^{-1}$ property.

	For the reverse implication, assume that $g$ has the Luzin $N^{-1}$ property. 
	We want to show that $\mu(\Sigma_g) = 0$.
	Since $g$ is $C^1$, by Sard's theorem \citep[Thm.~6.2]{milnor1963morse}, \citep[Thm.~6.10]{lee2012smoothmanifolds}, we know that $\mu(g(\Sigma_g)) = 0$.
	Therefore, as we assumed $g$ has the Luzin $N^{-1}$ property, we obtain that $g^{-1}(g(\Sigma_g))$ has measure zero.
	Then, since $\Sigma_g \subseteq g^{-1}(g(\Sigma_g))$, we obtain that $\Sigma_g$ also has measure zero, so we are done.
\end{proof}

\section{About measure zero sets} \label{appendix:proofs-measure-zero}

\begin{proof}[Proof of Lemma~\ref{lemma:vertical-horizontal-slices}]
	We show that the first two items are equivalent; the rest follows by symmetry.
	For a set $A$, let~$1_A$ be the indicator function of $A$.
	Since $S$ is a measurable set, the function $1_S$ is measurable, so from Tonelli's theorem \cite[Prop.~5.2.1]{cohn2013measure}, we have that  
	\begin{equation} \label{eq:tonelli}
		\mu_{n+m}(S)
		= \int_{\reals^n \times \reals^m} 1_S(x, y) \intd(x, y) 
		= \int_{\reals^n} \Big( \int_{\reals^m} 1_S(x, y) \intd y \Big) \intd x.
	\end{equation}
	At the same time, as $S$ is measurable, the slices $S^x$ are also measurable for all $x$ \cite[Lem.~5.1.2]{cohn2013measure} and their measure is given by 
	\begin{equation*}
		\mu_m(S^x) = 
		\int_{\reals^m} 1_{S^x}(y) \intd y = 
		\int_{\reals^m} 1_{S}(x, y) \intd y .
	\end{equation*}
	Combining this with Eq.~\eqref{eq:tonelli}, we obtain
	\begin{equation*}
		\mu_{n+m}(S) = 
		\int_{\reals^n} \mu_m(S^x) \intd x.
	\end{equation*}
	Now, if $\mu_{n+m}(S) = 0$, then $\mu_m(S^x)$ vanishes $\mu_n$-almost everywhere \cite[Cor.~3.12]{cohn2013measure}.
	The other way around, if $\mu_m(S^x) = 0$ for almost any $x \in \Rn$, then $\mu_{n+m}(S) = 0$ by \cite[Ex.~2.3.7(d)]{cohn2013measure}.
\end{proof}

\arxivonly{\section{Proximal point algorithm on Hadamard manifolds} \label{app:proximalhadamard}

This appendix is included only in the arXiv version of this paper.
Recall Section~\ref{section:hadamard} for context about saddle avoidance on Hadamard manifolds.
The additional result proposed here parallels a Euclidean version stated in~\citep[\S5.2]{lee2019strictsaddles}.

\begin{theorem}
  Let $f \colon \calM \to \reals$ be $C^2$ on a Hadamard manifold $\calM$.
	Assume $\grad f$ is $L$-Lipschitz continuous.
	Consider the \emph{proximal point algorithm}
	\begin{equation*}
		x_{k+1} = \gpp(x_k) \coloneq \argmin{z \in \calM} \, f(z) + \frac{1}{2 \alpha} \dist(x_k, z)^2.
	\end{equation*}
	Then, for $0 < \alpha < 1/L$, iterating $\gpp$ avoids the strict saddle points of $f$.
\end{theorem}
\begin{proof}
	Let $x \in \calM$ and $n = \dim(\calM)$.
	Define $h_x \colon \calM \rightarrow \reals$, $h_x(z) = \frac{1}{2} \dist(z, x)^2$.
	Since $\calM$ is a Hadamard manifold, $\inj(\calM) = \infty$ \citep[Prop.~12.9]{lee2018riemannian} and so $h_x$ is smooth by Lemma~\ref{lemma:dist-smooth}.
	Define now a $C^2$ function $q_x \colon \calM \to \reals$ such that the iteration map becomes
	\begin{align*}
		\gpp(x) = \argmin{z \in \calM} \, q_x(z) && \textrm{ with } && q_x(z) = f(z) + \frac{1}{\alpha} h_x(z).
	\end{align*}
	Before checking if $\gpp$ has the Luzin $N^{-1}$ property, we need to argue that it is well defined.
	Without loss of generality, assume $L > 0$.
	
	First, note that for $0 < \alpha < 1/L$ the function $q_x$ is geodesically strongly convex.
	Indeed, let $\varepsilon = 1/L - \alpha > 0$.
	Since $\grad f$ is $L$-Lipschitz continuous, we have that $\alpha \| \Hess f(z) \| \leq 1 - \varepsilon L$.
	At the same time, $\Hess\,h_x(z) \succeq I$ \citep[Thm.~6.6.1]{jost2017riemannian}.
	Thus 
	\begin{equation} \label{eq:hess-q-x-pd}
		\alpha \Hess\,q_x(z) = \alpha \Hess f(z) + \Hess\,h_x(z) \succeq \varepsilon L \; I.
	\end{equation}
	Since $\calM$ is connected and complete, it is geodesically convex and so~\eqref{eq:hess-q-x-pd} implies that $q_x$ is geodesically strongly convex on $\calM$~\citep[Thm.~11.23]{boumal2020intromanifolds}.
	Therefore, $q_x$ admits exactly one (global) minimizer~\citep[Lem.~11.27]{boumal2020intromanifolds}.
	This confirms that the iteration map $\gpp$ is well defined.

	Now, let $y = \gpp(x)$.
	This is the unique point such that $\grad q_x(y) = 0$.
	Thus, $\alpha \grad f(y) = -\grad h_x(y) = \Log_y(x)$ and so $x = \Exp_y(\alpha \grad f(y))$.
	The right-hand side is nothing but the RGD iteration map for $-f$ at $y$, which we will denote by $\ggd(y)$.
	Therefore, we have 
	\begin{align} \label{eq:gpp-ggd-inverse}
		y = \gpp(x) = \gpp(\ggd(y)) 
		&& \textrm{ and } &&
		x = \ggd(y) = \ggd(\gpp(x)),
	\end{align}
	In the proof of Theorem~\ref{thm:hadamard-strict-saddle-avoidance}, under the same assumptions, we showed that $\D \ggd(y)$ is invertible for all $y \in \calM$. 
	By the inverse function theorem \cite[Thm.~4.5]{lee2012smoothmanifolds}, it follows that $\gpp$ is also $C^1$.
	Differentiating both sides of second identity in Eq.~\eqref{eq:gpp-ggd-inverse} with respect to $x$, we obtain that 
	\begin{equation} \label{eq:dg-pp-dg-gd}
		I = \D \ggd(\gpp(x)) \circ \D \gpp(x),
	\end{equation}
	which shows that $\D \gpp(x)$ is invertible for all $x$ when $0 < \alpha < 1/L$.
	By Theorem~\ref{thm:luzin-property}, this implies that $\gpp$ satisfies the Luzin $N^{-1}$ property. 

	It remains to show that strict saddles are unstable fixed points of $\gpp$.
	Let $x^\ast \in \calM$ be a strict saddle of $f$.
	Note that $\ggd (x^\ast) = \Exp_{x^\ast}(\alpha \grad f(x^\ast)) = x^\ast$ and so $\gpp(x^\ast) = x^\ast$, too, which confirms that $x^\ast$ is fixed for $\gpp$.
	From \eqref{eq:dg-pp-dg-gd} and Lemma~\ref{lemma:Dgx-critical}, it follows that 
	\begin{equation*}
		\D \gpp(x^\ast) = (\D \ggd(x^\ast))^{-1} = (I + \alpha \Hess f(x^\ast))^{-1}.
	\end{equation*}
	Since $\Hess f(x^\ast)$ has a negative eigenvalue in the interval $[-L, 0)$, we find that $\D \gpp(x^\ast)$ has an eigenvalue strictly greater than $1$, as intended.
	Conclude with Theorem~\ref{thm:saddle-avoidance-template}.
\end{proof}}

\end{document}

%% file: preamble.tex
\usepackage[colorlinks=true,bookmarks=true,linkcolor=blue,urlcolor=blue,citecolor=blue,breaklinks=true]{hyperref}
\usepackage{breakcites}

\usepackage{doi}

\usepackage{amsmath}
\usepackage{amsfonts}
\usepackage{amsthm}
\usepackage{amssymb}

\usepackage{mathtools}

\usepackage{xcolor}
\usepackage{bm} 
\usepackage{dsfont} 

\usepackage{graphicx}
\usepackage{tikz-cd}
\usepackage{tikz}
\definecolor{mycolor1}{rgb}{0.105882,0.619608,0.466667}
\definecolor{mycolor2}{rgb}{0.85098,0.372549,0.00784314}
\definecolor{mycolor3}{rgb}{0.458824,0.439216,0.701961}
\definecolor{mycolor4}{rgb}{0.905882,0.160784,0.541176}
\definecolor{mycolor5}{rgb}{0.4,0.65098,0.117647}
\definecolor{mycolor6}{rgb}{0.65098,0.462745,0.113725}
\definecolor{mycolor7}{rgb}{0.901961,0.670588,0.00784314}
\definecolor{mycolor8}{rgb}{0.4,0.4,0.4}
\definecolor{mycolor9}{rgb}{0.301961,0,0.294118}
\definecolor{mycolor10}{rgb}{0.0313725,0.25098,0.505882}

\newcommand{\arccot}{\operatorname{arccot}}

\usepackage{psfrag}

\usepackage{algorithm}
\usepackage{algpseudocode}
\algdef{SE}[DOWHILE]{Do}{doWhile}{\algorithmicdo}[1]{\algorithmicwhile\ #1}%

\usepackage{mathrsfs} 

\newcommand{\diam}{\mathrm{diam}}

\newcommand{\argmin}[1]{\underset{#1}{\operatorname{argmin}}}

\newcommand{\inner}[2]{\left\langle{#1},{#2}\right\rangle}

\newcommand{\trace}{\mathrm{Tr}}

\newcommand{\Proj}{\mathrm{Proj}}

\newcommand{\Exp}{\mathrm{Exp}}
\newcommand{\Log}{\mathrm{Log}}

\newcommand{\Retr}{\mathrm{R}}
\newcommand{\inj}{\mathrm{inj}}

\newcommand{\St}{\mathrm{St}}

\newcommand{\T}{\mathrm{T}}
\newcommand{\N}{\mathrm{N}}

\newcommand{\Rk}{{\mathbb{R}^{k}}}

\newcommand{\reals}{{\mathbb{R}}}
\newcommand{\naturals}{\mathbb{N}}

\newcommand{\Rn}{{\mathbb{R}^n}}
\newcommand{\Rm}{{\mathbb{R}^m}}

\newcommand{\grad}{\mathrm{grad}}
\newcommand{\Hess}{\mathrm{Hess}}

\newcommand{\D}{\mathrm{D}}

\newcommand{\dds}{\frac{\mathrm{d}}{\mathrm{d}s}}

\newcommand{\calE}{\mathcal{E}}

\newcommand{\calK}{\mathcal{K}}

\newcommand{\calT}{\mathcal{T}}

\newcommand{\calM}{\mathcal{M}}
\newcommand{\calN}{\mathcal{N}}

\newcommand{\calP}{\mathcal{P}}
\newcommand{\rank}{\operatorname{rank}}

\newcommand{\Kmax}{K_{\mathrm{max}}}

\newcommand{\Sn}{\mathbb{S}^{n-1}}

\newcommand{\intd}{\mathrm{d}}

\newcommand{\TODOF}[1]{}

\newcommand{\dist}{\mathrm{dist}}

\newcommand{\Wcsloc}{W_{\mathrm{cs}}^\mathrm{loc}}
\newcommand{\gpp}{g_{\mathrm{PP}}}
\newcommand{\ggd}{g_{\mathrm{GD}}}

\newtheorem{theorem}{Theorem}
\newtheorem{lemma}{Lemma}
\newtheorem{proposition}{Proposition}
\newtheorem{corollary}{Corollary}

\newtheorem{definition}{Definition}
\newtheorem{remark}{Remark}

\usepackage[lmargin=1in,rmargin=1in,bottom=1in,top=1in]{geometry}
\usepackage[ansinew]{inputenc}
\usepackage[round,compress]{natbib}

\usepackage{sectsty}
\makeatletter\def\@seccntformat#1{\protect\makebox[0pt][r]{\csname the#1\endcsname\hspace{12pt}}}\makeatother

\usepackage{subcaption}

\usepackage{enumitem} 

\usepackage[normalem]{ulem} 

\usepackage{multicol}
\usepackage{multirow}

\usepackage{blindtext} 
\usepackage{soul}

\newcommand{\aref}[1]{\hyperref[#1]{A\ref{#1}}}

\newif\ifarxiv
\arxivtrue   

\newcommand{\arxivonly}[1]{\ifarxiv #1\fi}